\documentclass[11pt,reqno]{amsart}

\usepackage[margin=1.15in]{geometry}
\usepackage[T1]{fontenc}
\usepackage[utf8]{inputenc}
\usepackage{lmodern}
\usepackage{microtype}

\usepackage{amsmath,amssymb,amsfonts,amsthm,mathtools}
\usepackage{mathrsfs}
\usepackage{enumitem}

\usepackage{xcolor}
\usepackage{graphicx}
\usepackage{comment}
\usepackage{tikz-cd}
\usepackage{multicol}
\usepackage[colorlinks=true,
linkcolor=blue!50!black,
citecolor=blue!50!black,
urlcolor=blue!50!black]{hyperref}
\usepackage[nameinlink,capitalize]{cleveref}

\theoremstyle{plain}
\newtheorem{theorem}{Theorem}[section]
\newtheorem{thm}[theorem]{Theorem}
\newtheorem{thmsn}[theorem]{Theorem}

\newtheorem{prop}[theorem]{Proposition}
\newtheorem{propsn}[theorem]{Proposition}
\newtheorem{lemma}[theorem]{Lemma}

\newtheorem{coro}[theorem]{Corollary}

\newtheorem{fact}[theorem]{Fact}

\theoremstyle{definition}
\newtheorem{definition}[theorem]{Definition}
\newtheorem{example}[theorem]{Example}

\theoremstyle{remark}

\newtheorem{rmk}[theorem]{Remark}

\theoremstyle{plain}
\newtheorem{mainthm}{Theorem}

\numberwithin{equation}{section}

\newcommand{\R}{\mathbb{R}}
\newcommand{\N}{\mathbb{N}}
\newcommand{\Z}{\mathbb{Z}}

\newcommand{\T}{\mathbb{T}}
\newcommand{\A}{\mathbb{A}}

\newcommand{\M}{\mathcal{M}}

\newcommand{\Leb}{\operatorname{Leb}}
\newcommand{\Per}{\operatorname{Per}}

\newcommand{\diam}{\operatorname{diam}}

\newcommand{\Id}{\operatorname{Id}}

\DeclareMathOperator{\interior}{int}
\newcommand{\proj}{\operatorname{proj}}

\newcommand{\restr}[2]{#1|_{#2}}
\newcommand{\p}{\partial}

\newcommand{\eps}{\varepsilon}

\newcommand{\qand}{\qquad}
\title[A periodic flow with high emergence]{A periodic flow with high emergence}

\author{Odylo Costa}
\address{Departamento de Matemática, Universidade Federal do Cear\'a, Fortaleza, Brazil}
\email{odylo.costa@mat.ufc.br}

\subjclass[2020]{37A35, 37C10, 37C40, 37C86, 57R30}
\keywords{Periodic flows, emergence, compact foliations, periodic orbit conjecture, Wasserstein distance}

\date{\today}

\begin{document}

\begin{abstract}
We construct a smooth nonsingular periodic flow on a compact manifold with high emergence, in sharp contrast with the low statistical complexity of periodic self-maps. The construction is based on a modification of the Epstein--Vogt counterexample to the Periodic Orbit Conjecture and on the high-emergence mechanism of Berger--Bochi.
\end{abstract}

\maketitle


\section{Introduction}

Given a continuous dynamical system on a compact metric space, a fundamental problem is to understand its statistical behavior and how complex it can be. One way to describe the statistics of a system is through its empirical measures. For a map $f\colon M\to M$, these are the probability measures
\begin{equation*}
\mathsf e_n^f(x)
=
\frac1n\sum_{j=0}^{n-1}\delta_{f^j(x)},
\end{equation*}
which record the proportion of time that the orbit of $x$ spends in each region of the phase space. When these measures converge, their limit describes the asymptotic statistics of the orbit. Thus a natural question is: at a given precision $\varepsilon>0$, how many probability measures are needed to approximate the statistical behavior of most orbits? The notion of emergence, introduced by Berger in \cite{berger2017emergence}, gives a quantitative approach to this question by measuring the growth of this number as $\varepsilon\to0$.

More precisely, if $f\colon M\to M$ is a continuous map and $\mu$ is a reference probability measure, the emergence $\mathcal E_\mu(f)(\varepsilon)$ is the smallest integer $N$ for which there exist probability measures $\mu_1,\dots,\mu_N$ such that
\begin{equation*}
\limsup_{n\to\infty}
\int_M
\min_{1\le i\le N}
\mathsf d\bigl(\mathsf e_n^f(x),\mu_i\bigr)\,d\mu(x)
\le \varepsilon,
\end{equation*}
where $\mathsf d$ is the Kantorovich--Wasserstein distance on the space of probability measures. A similar definition applies to flows by replacing the discrete empirical measures by time-averaged empirical measures.

Emergence gives a way to classify systems by the complexity of their statistical behavior. We say a system has finite emergence with respect to $\mu$ if
\begin{equation*}
        \mathcal E_\mu(f)(\varepsilon)=O(1)
\end{equation*}
as $\varepsilon\to0$, and it has at most polynomial emergence if there exists
$\alpha\ge0$ such that
\begin{equation*}
        \mathcal E_\mu(f)(\varepsilon)=O(\varepsilon^{-\alpha})
\end{equation*}
as $\varepsilon\to0$. On the other hand, one says that a system has
superpolynomial emergence if
\begin{equation*}
        \limsup_{\varepsilon\to0}
        \frac{\log \mathcal E_\mu(f)(\varepsilon)}
        {-\log\varepsilon}
        =
        +\infty.
\end{equation*}

To understand better systems with superpolynomial emergence, Berger introduced in
\cite{berger2020complexities} the notion of order of emergence:
\begin{equation*}
        \overline{\mathcal O}\mathcal E_\mu(f)
        =
        \limsup_{\varepsilon\to0}
        \frac{\log\log \mathcal E_\mu(f)(\varepsilon)}
        {-\log\varepsilon}.
\end{equation*}
If a system has at most polynomial emergence, then its order of emergence is
zero. Conversely, positive order of emergence implies superpolynomial
emergence. Thus positive order is a strong form of high statistical complexity.
In this paper, the expression high emergence will refer to this positive-order
regime.

The existence of systems with high emergence is a central theme in the recent development of the theory. The first example of high emergence was given by Berger and Bochi in \cite{bergerbochi}, where they constructed smooth conservative systems with such property. Further examples and mechanisms have since appeared in different contexts, including polynomial automorphisms of the plane \cite{berger2023emergence}, smooth dissipative systems \cite{kiriki2022emergence}, and analytic pseudo-rotations \cite{berger22,delaporte2025analytic}. 

The aim of this paper is to exhibit high emergence in a class of systems that are, from the topological point of view, as simple as possible: periodic systems. For pointwise periodic homeomorphisms this is impossible: by Montgomery's theorem \cite{montgomery}, if every point of a homeomorphism of a connected manifold is periodic, then the periods are uniformly bounded; consequently, as we will show, such a system has at most polynomial emergence. For flows, however, the situation is radically different. Counterexamples to the Periodic Orbit Conjecture show that there exist smooth nonsingular flows on compact manifolds whose orbits are all periodic but whose periods are unbounded.

Our goal is to show that this phenomenon can carry genuinely rich statistical behavior. More precisely, we modify the Epstein--Vogt counterexample to the Periodic Orbit Conjecture \cite{EV78} in order to construct a smooth nonsingular periodic flow on a compact $4$-manifold with high metric emergence. The construction proceeds by deforming the Hamiltonian base of the Epstein--Vogt model so as to insert, as factor, the high-emergence mechanism of Berger--Bochi. The resulting flow still has only periodic orbits, but the dynamics is complex from the statistical point of view.

\begin{mainthm}\label{thm:main-lower-bound}
There exist a smooth nonsingular periodic
flow $(\Phi^t)_{t\in\mathbb R}$ on a compact smooth $4$-manifold $M$ such that, for every $t\neq0$,
\begin{equation*}
\overline{\mathcal O}\mathcal E_\mathrm{Leb}(\Phi^t)\ge 2.
\end{equation*}
\end{mainthm}

\subsection{Organization of the paper}

The paper is organized as follows. In Section~\ref{section: emergence}, we recall the definition of emergence for self-maps, extend it to flows, and give some examples of systems with low emergence. Moreover, we recall the relation of emergence and quantization of measures, and study how emergence behaves under Lipschitz and H\"older factors. 

In Section~\ref{section: Epstein-Vogt's counterexample} we recall and generalize the Epstein--Vogt counterexample for the Periodic Orbit Conjecture: we construct a family of smooth flows on compact manifolds such that every orbit is periodic and the periods are unbounded. 

Finally, in Section~\ref{subsection: An example of foliation with positive order of emergence}, we modify the Epstein--Vogt construction in order to obtain a periodic flow with high emergence. The modification is performed on the Hamiltonian base and uses a factor of the high-emergence mechanism of Berger--Bochi.

\section*{Acknowledgements}

This work has received funding from the European Union's Horizon 2020 research
and innovation programme under the Marie Sk{\l}odowska-Curie grant agreement
No.~945332. The author was also supported by the Instituto Serrapilheira grant
``Jangada Dinâmica: Impulsionando Sistemas Dinâmicos na Região Nordeste''. The
author is grateful to Pierre Berger for his advice and guidance in the
preparation of this paper, and to C\'esar Romero Mora for helpful corrections.

\section{Emergence}\label{section: emergence}

\subsection{Emergence of self-maps} 
Given a continuous map $f\colon X\to X$ on a compact metric space $(X,d)$, we want to study the statistical behavior of the orbits of $f$. To do so, for each $x\in X$ and $n\ge 1$, we can associate to an orbit segment $(f^{i}(x))_{i=0}^{n-1}$ the $n$-th \textit{empirical measure associated to} $x$: $$\mathsf{e}_{n}^{f}(x)\coloneqq \frac{1}{n}\sum_{i=0}^{n-1}\delta_{f^{i}(x)}.$$

Fix an invariant measure $\mu$ for $f$, i.e., $f_{\ast}\mu=\mu$. By Birkhoff's Ergodic Theorem, for $\mu$-a.e. $x\in X$, the sequence $(\mathsf{e}_{n}^{f}(x))_{n}$ converges to a unique measure: $$\mathsf{e}^{f}(x)\coloneqq \lim_{n\to \infty} \mathsf{e}_{n}^{f}(x),$$ which we call the \textit{empirical measure associated to} $x$. 

The measure $\mathsf{e}^{f}(x)$ describes the statistical behavior of the orbit of $x$. The \textit{empirical function} $\mathsf{e}^{f}\colon X \to \mathcal{M}(X)$, which maps $X$ to the space of probability measures on $X$, is measurable. Moreover, $\mathsf{e}^{f}$ is $\mu$-a.e. constant if, and only if, $\mu$ is ergodic with respect to $f$. In this case, $\mathsf{e}^{f}(x)=\mu$ to $\mu$-a.e. $x\in X$. 

Let $\mathcal{M}_{f}(X)$ be the closed convex subset of probability measures on $X$ that are invariant under $f$. As a consequence of the ergodic decomposition theorem, for every $\mu\in \mathcal{M}_{f}(X)$ there exists a Borel set $X_{0}\subset X$ with full probability, i.e., $\mu(X_{0})=1$, such that for every $x\in X_{0}$, the empirical measure $\mathsf{e}^{f}(x)$ is $f$-invariant and ergodic. So, for any $\mu\in \mathcal{M}_{f}(X)$, the measure $\mathsf{e}^{f}_{\ast}(\mu)\in \mathcal{M}(\mathcal{M}(X))$ gives full weight to the set of ergodic measures $\mathcal{M}_{f}^{\mathrm{erg}}(X)\subset \mathcal{M}_{f}(X)$. The probability measure $\mathsf{e}^{f}_{\ast}(\mu)$ is called the \textit{ergodic decomposition} of $\mu$ and it describes the distribution of the statistical behaviors of the orbits.

To measure the size of the ergodic decomposition, we endow $\mathcal{M}(X)$ with a metric $\mathsf{d}$, called the Kantorovich--Wasserstein distance, defined by: \begin{equation} \label{def: katorovitch-wasserstein} \mathsf{d}(\mu,\nu)=\inf_{\pi \in \Pi(\mu,\nu)}\int_{X\times X} d(x,y) \ d\pi(x,y),\end{equation} where $\Pi(\mu,\nu)\coloneqq \left\{\pi \in \mathcal{M}(X\times X) \mid (p_{1})_{\ast}\pi=\mu \text{ and }(p_{2})_{\ast}\pi=\nu\right\}$ and $p_{i}\colon X\times X \to X$ are the projections onto the first and second coordinates. This distance induces the weak-$\star$ topology on $\mathcal{M}(X)$ (Corollary 6.13 in \cite{villani2009optimal}), which is compact. 

Equivalently, since $X$ is compact, the Kantorovich--Wasserstein distance $\mathsf{d}$ can be defined as \begin{equation} \label{def: katorovitch-wasserstein 2} \mathsf{d}(\mu,\nu)=\sup_{u\in \mathrm{Lip}_{1}} \left| \int_{X} u \ d\mu - \int_{X} u \ d\nu \right|.\end{equation}

A natural question arises when $M$ is a compact manifold and $\mu=\mathrm{Leb}$: in this setting, are there dynamics $f$ for which $\mathsf{e}^{f}_{\ast}\mathrm{Leb}$ is big, i.e., the statistical behavior of the orbits is highly complex? To state this question precisely, we recall the definition of \textbf{emergence}.

\begin{definition}[Emergence]\label{def: metric emergence}
	Let $(X,d)$ be a compact metric space endowed with a probability measure $\mu$, and let $f\colon X \to X$ be a continuous map on $X$ (not necessarily $\mu$-preserving). 
	
	The \textit{emergence} $\mathcal{E}_{\mu}(f)$ is the function that, for each scale $\eps>0$, associates the minimal number $\mathcal{E}_{\mu}(f)(\eps)=N\geq 1$ of probability measures $\mu_{1},\dots,\mu_{N}$ such that: $$\limsup_{n\to \infty}{\int \min_{1\leq i\leq N}{\mathsf{d}(\mathsf{e}_{n}^{f}(x),\mu_{i})} \ d\mu(x)}\leq \eps,$$ where $\mathsf{d}$ is the Kantorovich--Wasserstein distance defined in $(\ref{def: katorovitch-wasserstein})$. 
\end{definition}

Observe that when $\mu$ is $f$-invariant, the sequence $(\mathsf{e}_{n}^{f}(x))_{n}$ converges $\mu$-a.e. to $\mathsf{e}^{f}(x)$, and the above expression can be replaced by $$\int \min_{1\leq i\leq N}{\mathsf{d}(\mathsf{e}^{f}(x),\mu_{i})} \ d\mu(x)\leq \eps.$$

\subsection{Emergence of flows} 

For a flow $(\varphi^{t})_{t}$ on $X$, one defines the analogue of the $n$-th empirical measure as follows. For $T>0$ set
\[
\mathrm{e}^{\varphi}_{T}(x)\coloneqq \frac{1}{T}\int_{0}^{T}\delta_{\varphi^{s}(x)}\,ds.
\]
If $\mu$ is invariant (i.e.\ $(\varphi^{t})_{\ast}\mu=\mu$ for all $t\in\mathbb R$), then the limit
\[
\mathrm{e}^{\varphi}(x)\coloneqq \lim_{T\to +\infty}\frac{1}{T}\int_{0}^{T}\delta_{\varphi^{s}(x)}\,ds
\]
exists $\mu$-a.e. and is called the empirical measure associated to the flow $(\varphi^{t})_{t}$.

\begin{definition}[Emergence of a flow]\label{def: metric emergence of a flow}
	Let $(X,d)$ be a compact metric space endowed with a probability measure $\mu$, and let $(\varphi^{t})_{t}$ be a continuous flow on $X$ (not necessarily $\mu$-preserving). 
	
	The \textit{emergence} $\mathcal{E}_{\mu}(\varphi)$ is the function that assigns to each scale $\varepsilon>0$ the minimal number $\mathcal{E}_{\mu}(\varphi)(\varepsilon)=N\ge 1$ of probability measures $\mu_{1},\dots,\mu_{N}$ such that
	\begin{equation}\label{def: emergence flow limsup}
		\limsup_{T\to +\infty}\int \min_{1\le i\le N}\mathsf{d}(\mathsf{e}_{T}^{\varphi}(x),\mu_{i})\ d\mu(x)\le \varepsilon,
	\end{equation}
	where $\mathsf{d}$ is the Kantorovich--Wasserstein distance defined in \eqref{def: katorovitch-wasserstein}.
	\end{definition}

As in the discrete-time case, if $\mu$ is invariant under the flow $(\varphi^{t})_{t}$, then $\mathsf{e}_{T}^{\varphi}(x)$ converges $\mu$-a.e. to $\mathsf{e}^{\varphi}(x)$, and the condition \eqref{def: emergence flow limsup} can be replaced by
\begin{equation}\label{def: emergence flow}
	\int \min_{1\le i\le N}\mathsf{d}(\mathsf{e}^{\varphi}(x),\mu_{i})\ d\mu(x)\le \varepsilon.
\end{equation}

\begin{rmk}\label{rmk:periodic-flow-empirical-convergence}
Note that if $(\varphi^{t})_{t}$ is a continuous flow whose every orbit is
periodic, then one does not need to assume that $\mu$ is invariant in order to
replace \eqref{def: emergence flow limsup} by \eqref{def: emergence flow}.

Indeed, fix $x\in X$ and let $P=\Per(x)$. Define
\begin{equation*}
\mathsf e^\varphi(x)
=\frac{1}{P}\int_0^P \delta_{\varphi^s(x)}\,ds.
\end{equation*}
Given $T>0$, write
\begin{equation*}
T=NP+r,
\qquad
0\leq r<P.
\end{equation*}
Then
\begin{equation*}
\mathsf e_T^\varphi(x)
=\frac{NP}{T}\mathsf e^\varphi(x)
+
\frac{1}{T}\int_0^r\delta_{\varphi^s(x)}\,ds.
\end{equation*}
Since
\begin{equation*}
\mathsf e^\varphi(x)
=
\frac{NP}{T}\mathsf e^\varphi(x)
+
\frac{r}{T}\mathsf e^\varphi(x),
\end{equation*}
and since the diameter of $\mathcal{M}(X)$ is bounded by $\diam(X)$,
we can estimate
\begin{align*}
\mathsf d(\mathsf e_T^\varphi(x),\mathsf e^\varphi(x))
&=
\mathsf d\left(
\frac{NP}{T}\mathsf e^\varphi(x)
+
\frac{1}{T}\int_0^r\delta_{\varphi^s(x)}\,ds,
\frac{NP}{T}\mathsf e^\varphi(x)
+
\frac{r}{T}\mathsf e^\varphi(x)
\right) \\
&\leq
\frac{r}{T}\operatorname{diam}(X)
\leq
\frac{P}{T}\operatorname{diam}(X).
\end{align*}

Therefore
\begin{equation*}
\mathsf e_T^\varphi(x)\longrightarrow \mathsf e^\varphi(x)
\qquad\text{for every }x\in X.
\end{equation*}

Now fix probability measures $\mu_1,\dots,\mu_N$ on $X$. Then
\begin{equation*}
\min_{1\leq i\leq N}
\mathsf d(\mathsf e_T^\varphi(x),\mu_i)
\longrightarrow
\min_{1\leq i\leq N}
\mathsf d(\mathsf e^\varphi(x),\mu_i)
\end{equation*}
for every $x\in X$. Since these functions are bounded by
$\operatorname{diam}(X)$, dominated convergence gives
\begin{equation*}
\lim_{T\to+\infty}
\int_X
\min_{1\leq i\leq N}
\mathsf d(\mathsf e_T^\varphi(x),\mu_i)\,d\mu(x)
=\int_X
\min_{1\leq i\leq N}
\mathsf d(\mathsf e^\varphi(x),\mu_i)\,d\mu(x).
\end{equation*}
Thus, in the periodic case, emergence can be computed using the measures $\mathsf e^\varphi(x)$ supported on the orbit, even if the reference measure $\mu$ is not
invariant.
\end{rmk}

\subsection{Superpolynomial emergence}

In the search for systems with very complex statistical behavior and to better understand them, Berger introduced in \cite{berger2020complexities} the \emph{order of emergence} 
\[
\overline{\mathcal{O}}\mathcal{E}_{\mu}(g)=\limsup_{\varepsilon\to 0}\frac{\log\log \mathcal{E}_{\mu}(g)(\varepsilon)}{-\log \varepsilon},
\]
for $g$ a continuous self-map or a flow on $(X,\mu)$. 

Another complexity concept that will be useful to define is:

\begin{definition}[Polynomial Emergence] \label{def: polynomial emergence}
	We say a system $(g,\mu)$ has \textbf{at most polynomial emergence} if there exists $\alpha \geq 0$ such that $$\mathcal{E}_{\mu}(g)(\eps)=\mathcal{O}( \eps^{-\alpha})$$ when $\eps \to 0$. 
\end{definition}

In particular, if a system has polynomial emergence, then the order of emergence will be zero. More precisely, a system $(g,\mu)$ has at most polynomial emergence if and only if
	\[
	\limsup_{\varepsilon\to 0}\frac{\log \mathcal{E}_{\mu}(g)(\varepsilon)}{-\log\varepsilon}<\infty.
	\]

\subsection{Some examples of systems with low emergence}
\label{subsection: examples of low emergence}

In this subsection we record some elementary examples of systems with low
emergence. Throughout, $M$ is a closed Riemannian manifold of dimension $k$.
We denote by $d$ the Riemannian distance on $M$, and we fix
the normalized Riemannian volume, denoted by $\mathrm{Leb}$. We denote by
$\mathsf d$ the Kantorovich--Wasserstein distance on $\mathcal M(M)$.

We start with the simplest systems from the point of view of emergence:
ergodic ones.

\begin{example}[Ergodic systems]\label{ex:ergodic-systems}
Let $f:M\to M$ preserve $\mathrm{Leb}$, and assume that $\mathrm{Leb}$ is
ergodic. Then, by Birkhoff's ergodic theorem,
\begin{equation*}
\mathsf e^f(x)=\mathrm{Leb}
\qquad\text{for }\mathrm{Leb}\text{-a.e. }x.
\end{equation*}
Thus the empirical-measure map is $\mathrm{Leb}$-a.e. constant. Consequently,
\begin{equation*}
\mathcal E_{\mathrm{Leb}}(f)(\varepsilon)=1
\qquad\text{for every }\varepsilon>0.
\end{equation*}
In particular,
\begin{equation*}
\overline{\mathcal O}\mathcal E_{\mathrm{Leb}}(f)=0.
\end{equation*}
\end{example}

For the next examples we use the following elementary fact. Since $M$ is a
closed $k$-dimensional Riemannian manifold, there exists a constant
$C_M>0$ such that, for every sufficiently small $\varepsilon>0$, one can find
an $\varepsilon$-dense subset
\begin{equation*}
\{a_1,\dots,a_N\}\subset M
\end{equation*}
with
\begin{equation*}
N\leq C_M\varepsilon^{-k}.
\end{equation*}

Next, we consider an example with very simple dynamics but with a continuum of
distinct empirical measures.

\begin{example}\label{ex:identity-map}
Let $f=\Id_M$. Then, for every $x\in M$,
\begin{equation*}
\mathsf e^{\Id}(x)=\delta_x.
\end{equation*}
Let $\{a_1,\dots,a_N\}$ be an $\varepsilon$-dense subset of $M$ with
$N\leq C_M\varepsilon^{-k}$. For every $x\in M$, choose $a_i$ such that
\begin{equation*}
d(x,a_i)<\varepsilon.
\end{equation*}
Then
\begin{equation*}
\mathsf d(\delta_x,\delta_{a_i})=d(x,a_i)
<\varepsilon.
\end{equation*}
Therefore the family $\delta_{a_1},\dots,\delta_{a_N}$ $\varepsilon$-approximates
all empirical measures of the identity. Hence
\begin{equation*}
\mathcal E_{\mathrm{Leb}}(\Id_M)(\varepsilon)
\leq C_M\varepsilon^{-k}
\end{equation*}
for every sufficiently small $\varepsilon>0$. In particular,
\begin{equation*}
\overline{\mathcal O}\mathcal E_{\mathrm{Leb}}(\Id_M)=0.
\end{equation*}
\end{example}

\begin{example}[Periodic maps]\label{ex:periodic-systems}
Suppose that $f:M\to M$ is a continuous map and that there exists $q\geq 1$
such that
\begin{equation*}
f^q=\Id_M.
\end{equation*}
Then $f$ has at most polynomial emergence. More precisely, there exists a
constant $C>0$ such that
\begin{equation*}
\mathcal E_{\mathrm{Leb}}(f)(\varepsilon)
\leq C\,\varepsilon^{-kq}
\end{equation*}
for every sufficiently small $\varepsilon>0$. In particular,
\begin{equation*}
\overline{\mathcal O}\mathcal E_{\mathrm{Leb}}(f)=0.
\end{equation*}

Indeed, for every $x\in M$, the empirical measures of $f$ converge to
\begin{equation*}
E_f(x)
:=
\frac1q\sum_{j=0}^{q-1}\delta_{f^j(x)}.
\end{equation*}
To see this, write $n=\ell q+r$, with $0\leq r<q$. Then
\begin{align*}
\mathsf e_n^f(x)
&=
\frac1n\sum_{j=0}^{n-1}\delta_{f^j(x)} 
=\frac{\ell q}{n}
\left(
\frac1q\sum_{j=0}^{q-1}\delta_{f^j(x)}
\right)
+
\frac1n\sum_{j=0}^{r-1}\delta_{f^j(x)}.
\end{align*}
Since $\ell q/n\to 1$ and $r/n\to 0$, it follows that
\begin{equation*}
\mathsf e_n^f(x)\longrightarrow E_f(x)
\qquad\text{for every }x\in M.
\end{equation*}

Now fix $\varepsilon>0$, and let $\{a_1,\dots,a_N\}$ be an $\varepsilon$-dense
subset of $M$ with
\begin{equation*}
N\leq C_M\varepsilon^{-k}.
\end{equation*}
For each ordered $q$-tuple $(i_0,\dots,i_{q-1})$, with
$1\leq i_j\leq N$, define
\begin{equation*}
\nu_{i_0,\dots,i_{q-1}}
:=
\frac1q\sum_{j=0}^{q-1}\delta_{a_{i_j}}.
\end{equation*}
There are at most $N^q$ measures of this form.

We claim that this finite family $\varepsilon$-approximates all empirical
measures $E_f(x)$. Given $x\in M$, choose, for each $0\leq j\leq q-1$, an
index $i_j$ such that
\begin{equation*}
d(f^j(x),a_{i_j})<\varepsilon.
\end{equation*}
Then
\begin{align*}
\mathsf d\left(
E_f(x),
\nu_{i_0,\dots,i_{q-1}}
\right)
&=
\mathsf d\left(
\frac1q\sum_{j=0}^{q-1}\delta_{f^j(x)},
\frac1q\sum_{j=0}^{q-1}\delta_{a_{i_j}}
\right) \leq
\frac1q\sum_{j=0}^{q-1}
d(f^j(x),a_{i_j}) <\varepsilon.
\end{align*}
The inequality follows from the transport plan which sends the mass $1/q$ at
$f^j(x)$ to the mass $1/q$ at $a_{i_j}$.

Thus every limiting empirical measure is $\varepsilon$-close to one of the
measures $\nu_{i_0,\dots,i_{q-1}}$. Consequently,
\begin{equation*}
\mathcal E_{\mathrm{Leb}}(f)(\varepsilon)
\leq N^q
\leq C_M^q\varepsilon^{-kq}.
\end{equation*}
This proves the polynomial upper bound.
\end{example}

\begin{rmk}[Pointwise periodic homeomorphisms]\label{rmk:montgomery-pointwise-periodic}
The previous example applies, in particular, to pointwise periodic
homeomorphisms of compact manifolds. Indeed, by Montgomery's theorem
\cite{montgomery}, if $f$ is a homeomorphism of a connected manifold and every
point is periodic, then $f$ is periodic: there exists $q\geq 1$ such that
\begin{equation*}
f^q=\Id.
\end{equation*}

If $M$ is compact but not connected, then $M$ has finitely many connected
components. Since $f$ is a homeomorphism, it permutes these components. Hence
there exists $s\geq 1$ such that $f^s$ preserves each connected component.
Applying Montgomery's theorem to $f^s$ on each component and taking a common
multiple of the resulting periods, one obtains a global period for $f$:
there exists $Q\geq 1$ such that
\begin{equation*}
f^Q=\Id_M.
\end{equation*}
Consequently, every pointwise periodic homeomorphism of a compact manifold has
zero order of emergence.

This is a basic difference between maps and flows on compact manifolds. For
maps, pointwise periodicity forces a uniform global period. For flows, all
orbits may be periodic while the periods are unbounded; this is precisely the
phenomenon exploited by the Epstein--Vogt construction.
\end{rmk}

Finally, we present a family of examples which appear naturally in hyperbolic
dynamics and which, although dynamically interesting, are still simple from the
point of view of emergence.

\begin{example}[Finitely many attracting periodic orbits]\label{ex:finite-sinks}
Assume that $f:M\to M$ has finitely many attracting periodic orbits
\begin{equation*}
\mathcal O_1,\dots,\mathcal O_k
\end{equation*}
whose basins cover $\mathrm{Leb}$-almost every point of $M$. Let $\nu_i$ be the
uniform probability measure on the periodic orbit $\mathcal O_i$. If $B_i$
denotes the basin of $\mathcal O_i$, then
\begin{equation*}
\mathsf e^f(x)=\nu_i
\qquad\text{for }\mathrm{Leb}\text{-a.e. }x\in B_i.
\end{equation*}
Therefore the set of possible empirical measures has cardinality at most $k$.
Hence
\begin{equation*}
\mathcal E_{\mathrm{Leb}}(f)(\varepsilon)\leq k
\qquad\text{for every }\varepsilon>0.
\end{equation*}
In particular, systems with finitely many attracting periodic orbits whose
basins cover full volume have bounded emergence, and therefore zero order of
emergence. See \cite{berger2017emergence} for a more detailed explanation.
\end{example}

These examples illustrate that zero order of emergence is compatible with
non-ergodicity and even with a continuum of distinct empirical measures. What
matters for positive order of emergence is that the family of empirical
measures is large in a genuinely non-polynomial way.

\subsection{Quantization and ergodic decompositions} \label{subsection: emergence and quantization}

Again, let $(X,d)$ be a compact metric space, and let $\mathcal{M}(X)$ denote the space of probability measures on $X$, endowed with the Kantorovich--Wasserstein metric $\mathsf{d}$.

\begin{definition}
	The \textbf{quantization number} $\mathcal{Q}_{\mu}(\eps)$ of a measure $\mu \in \mathcal{M}(X)$ at scale $\eps>0$ is the least integer $N$ such that there exists a probability measure $\nu\in \mathcal{M}(X)$ supported on a set of cardinality $N$ satisfying $\mathsf{d}(\mu,\nu)\leq \varepsilon$.
\end{definition}

We now state a monotonicity property of the quantization number under a Lipschitz pushforward that will be useful below:

\begin{prop}[{\cite[Prop. 3.17]{bergerbochi}}]
	\label{quantizacao lipschitz}
	Let $(X,d)$ and $(Y,d)$ be compact metric spaces, $\mu \in \mathcal{M}(X)$ be a measure in $X$, and $h\colon X\to Y$ be a $\kappa$-Lipschitz map. Then, if $\nu\coloneqq h_{\ast}\mu$, we have: $$\mathcal{Q}_{\mu}(\eps)\geq \mathcal{Q}_{\nu}(\kappa\eps),$$ for all $\eps>0$.
\end{prop}

Whenever a dynamical system $f\colon X\to X$ preserves a measure $\mu$, we can relate the quantization number of its ergodic decomposition $\hat{\mu}\coloneqq\mathsf{e}^{f}_{\ast}\mu$ (viewed as a measure on $\mathcal{M}(X)$) to the emergence of $f$ with respect to~$\mu$:

\begin{prop}[{\cite[Prop. 3.12]{bergerbochi}}]
	\label{prop: quantizacao vs emeergencia}
	For every dynamics $f\colon X\to X$, the emergence of $f$ with respect to any invariant measure $\mu \in \mathcal{M}_{f}(X)$ equals the quantization number of the ergodic decomposition $\hat{\mu}\coloneqq \mathsf{e}^{f}_{\ast}\mu$: 
	\begin{equation}\label{eq: emergencia e quantizacao}
		\mathcal{E}_{\mu}(f)(\eps)=\mathcal{Q}_{\hat{\mu}}(\eps),
	\end{equation}
	for every $\eps>0$.
\end{prop}

\subsection{Factors and Emergence} \label{subsection: factors and emergence}

In this section we consider a continuous dynamical system $g$ on $(Y,\nu)$ which is a factor of a continuous dynamical system $f$ of $(X,\mu)$. This means that there exists $h: (X,\mu) \to (Y,\nu)$ such that the following diagram commutes:
	\[
         \begin{array} {rcccl}
	& &f& &\\
	& X  &\to & X &\\
	h&\downarrow &  &\downarrow &h\\
	& Y  &\to & Y &\\
	& &g& &\end{array}\qand h_{\ast} \mu = \nu\; . \]
	
We claim that we can bound the emergence of $g$ in terms of that of $f$, if  $h$ satisfies some regularity assumptions. To do so, we need a lemma (which was done in collaboration with Pierre Berger):

\begin{lemma}\label{l:push_Holder}
	Let $f \colon (Y,d) \to (Z,d)$ be a $(C,\alpha)$-H\"older map between compact metric spaces, with $0< \alpha\leq 1$.
	Let $F \colon (\M(Y),\mathsf{d}) \to (\M(Z),\mathsf{d})$ be the map $\mu \mapsto f_{\ast}\mu$.
	Then $F$ is $(C,\alpha)$-H\"older.
\end{lemma}
\begin{proof}
	Given $\mu_1$, $\mu_2 \in \M(Y)$, consider a transport plan $\pi \in \M(Y \times Y)$.
	Then $\tilde \pi \coloneqq (f \times f)_{\ast}(\pi)$ is a transport plan from $f_{\ast}\mu_1$ to $f_{\ast}\mu_2$ with:
	$$
	\mathrm{cost}(\tilde \pi) = \int d(f(x),f(y))\, d\pi(x,y) \le C \int d(x,y)^\alpha \, d\pi(x,y)\; .
	$$
	As the map $t\mapsto t^\alpha$ is concave down, by Jensen's inequality, we have:
	$$
	\mathrm{cost}(\tilde \pi) \le  C \left(\int d(x,y)\, d\pi(x,y)\right)^\alpha 
	= C \, \mathrm{cost}(\pi)^\alpha \, .
	$$
	So $\mathsf{d}(f_{\ast}\mu_1 , f_{\ast}\mu_2) \le C \cdot \mathsf{d}(\mu_1 , \mu_2)^\alpha$.
\end{proof}
						
\begin{prop} \label{lemma lipschitz preserva emergencia}
Let $f\colon X \to X$ and $g\colon Y \to Y$ be continuous maps between compact metric spaces. If there exists a $(C,\alpha)$-H\"older surjection $h\colon X \to Y$ such that $h\circ f=g\circ h$, then for every $\mu \in \mathcal{M}(X)$ and every $\eps>0$,  
$$\mathcal{E}_{\mu}(f)( \eps)\geq \mathcal{E}_{\nu}(g)(C\eps^\alpha),$$ 
where $\nu=h_{\ast}\mu$. In particular, if $f$ has at most polynomial emergence, $g$ also has it. 
\end{prop}
\begin{proof}
							Since $h$ is $(C,\alpha)$-H\"older, by Lemma \ref{l:push_Holder}, we have:
							$$\mathsf{d}(h_{\ast}\mathsf{e}_{n}^{f}(x),h_{\ast}\mu_{i})\leq C \cdot \mathsf{d}(\mathsf{e}_{n}^{f}(x),\mu_{i})^\alpha,$$ for any $x\in X$ and any $\mu_{i}\in \mathcal{M}(X)$. Since $h_{\ast}\mathsf{e}_{n}^{f}(x)=\mathsf{e}^{g}_{n}(h(x))$, we conclude  
							$$\mathsf{d}(\mathsf{e}^{g}_{n}(h(x)),h_{\ast}\mu_{i})\leq C \cdot \mathsf{d}(\mathsf{e}_{n}^{f}(x),\mu_{i})^\alpha.$$ 

							Hence given a finite family of probability measures $(\mu_i)_{1\le i \le N}$, we have:
							\begin{align*}
								\int_{Y} \min_{1\leq i\leq N}{\mathsf{d}(\mathsf{e}_{n}^{g}(y),h_{\ast}\mu_{i})} \ d\nu(y)& =   \int_{Y} \min_{1\leq i\leq N}{\mathsf{d}(\mathsf{e}_{n}^{g}(y),h_{\ast}\mu_{i})} \ dh_{\ast}\mu(y)\\
								& = \int_{X} \min_{1\leq i\leq N}{\mathsf{d}(\mathsf{e}_{n}^{g}(h(x)),h_{\ast}\mu_{i})} \ d\mu(x)\\
								&\le C \cdot   \int_{X} \min_{1\leq i\leq N} \mathsf{d}(\mathsf{e}_{n}^{f}(x),\mu_{i})^\alpha\ d\mu(x)\\
								&\le C \cdot  \left(  \int_{X} \min_{1\leq i\leq N} \mathsf{d}(\mathsf{e}_{n}^{f}(x),\mu_{i})\ d\mu(x)\right)^\alpha.
							\end{align*} 
							Now let $N=\mathcal{E}_{\mu}(f)(\varepsilon)$. By definition, we can choose
$\mu_1,\dots,\mu_N$ such that
\begin{equation*}
\limsup_{n\to\infty}
\int_{X}
\min_{1\leq i\leq N}
\mathsf{d}(\mathsf{e}_{n}^{f}(x),\mu_i)\,d\mu(x)
\leq \varepsilon.
\end{equation*}
Taking $\limsup_{n\to\infty}$ in the previous inequality, we obtain
\begin{equation*}
\limsup_{n\to\infty}
\int_{Y}
\min_{1\leq i\leq N}
\mathsf{d}(\mathsf{e}_{n}^{g}(y),h_{\ast}\mu_i)\,d\nu(y)
\leq C\varepsilon^\alpha.
\end{equation*}
Thus the measures $h_{\ast}\mu_1,\dots,h_{\ast}\mu_N$ approximate the
emergence of $g$ at scale $C\varepsilon^\alpha$. Hence
\begin{equation*}
\mathcal{E}_{\nu}(g)(C\varepsilon^\alpha)
\leq
N
=
\mathcal{E}_{\mu}(f)(\varepsilon),
\end{equation*}
and this proves the proposition.
\end{proof}

We have an analogous result for flows: 
						
\begin{prop} \label{lemma lipschitz preserva emergencia para fluxos}
Let $\varphi^{t}\colon X \to X$ and $\psi^{t}\colon Y \to Y$ be continuous flows between compact metric spaces. If there exists a $(C,\alpha)$-H\"older surjection $h\colon X \to Y$ such that $h\circ \varphi^{t}=\psi^{t}\circ h$ for every $t\in \R$, then for every $\mu \in \mathcal{M}(X)$ and every $\eps>0$, $$  \mathcal{E}_{\mu}(\varphi)( \eps)\geq \mathcal{E}_{\nu}(\psi)(C\eps^\alpha),$$ where $\nu=h_{\ast}\mu$. In particular, if $\varphi^{t}$ has at most polynomial emergence, $\psi^{t}$ also has it.
						\end{prop}
						\begin{proof}
							The proof follows the same lines as Proposition \ref{lemma lipschitz preserva emergencia}, with the only change that now we will have $$h_{\ast}\mathsf{e}_{T}^{\varphi}(x)=\mathsf{e}^{\psi}_{T}(h(x)),$$ for every $x\in X$ and every $T>0$.
						\end{proof}

We finish this section with a result about the emergence of a system with respect to two equivalent probability measures that will be useful to what follows:
						
\begin{prop}\label{prop: emergencia pol nao muda para medidas equiv}
Let $X$ be a compact metric space, and let $\mu$ and $\nu$ be probability
measures on $X$. Assume that there exists $K\geq 1$ such that
\begin{equation*}
K^{-1}\mu\leq \nu\leq K\mu.
\end{equation*}
Let $\varphi^{t}\colon X \to X$ be a continuous flow. Then, for every
$\varepsilon>0$,
\begin{equation*}
\mathcal{E}_{\nu}(\varphi)(K\varepsilon)
\leq
\mathcal{E}_{\mu}(\varphi)(\varepsilon)
\leq
\mathcal{E}_{\nu}(\varphi)(K^{-1}\varepsilon).
\end{equation*}
In particular, $\varphi$ has at most polynomial emergence with respect to
$\mu$ if and only if it has at most polynomial emergence with respect to
$\nu$.
\end{prop}

\begin{proof}
We prove the left inequality. Let
\begin{equation*}
N=\mathcal{E}_{\mu}(\varphi)(\varepsilon).
\end{equation*}
By definition, there exist probability measures
$\mu_1,\dots,\mu_N$ on $X$ such that
\begin{equation*}
\limsup_{T\to+\infty}
\int_X
\min_{1\leq i\leq N}
\mathsf d(\mathrm{e}^{\varphi}_{T}(x),\mu_i)\,d\mu(x)
\leq
\varepsilon.
\end{equation*}
Since $\nu\leq K\mu$, for every $T>0$ we have
\begin{equation*}
\int_X
\min_{1\leq i\leq N}
\mathsf d(\mathrm{e}^{\varphi}_{T}(x),\mu_i)\,d\nu(x)
\leq
K
\int_X
\min_{1\leq i\leq N}
\mathsf d(\mathrm{e}^{\varphi}_{T}(x),\mu_i)\,d\mu(x).
\end{equation*}
Taking $\limsup_{T\to+\infty}$, we get
\begin{equation*}
\limsup_{T\to+\infty}
\int_X
\min_{1\leq i\leq N}
\mathsf d(\mathrm{e}^{\varphi}_{T}(x),\mu_i)\,d\nu(x)
\leq
K\varepsilon.
\end{equation*}
Thus the same $N$ measures approximate the emergence with respect to $\nu$ at
scale $K\varepsilon$. Therefore
\begin{equation*}
\mathcal{E}_{\nu}(\varphi)(K\varepsilon)
\leq
N
=
\mathcal{E}_{\mu}(\varphi)(\varepsilon).
\end{equation*}

For the right inequality, one starts with
$N=\mathcal{E}_{\nu}(\varphi)(K^{-1}\varepsilon)$ and uses the other inequality
$\mu\leq K\nu$. This gives
\begin{equation*}
\mathcal{E}_{\mu}(\varphi)(\varepsilon)
\leq
\mathcal{E}_{\nu}(\varphi)(K^{-1}\varepsilon).
\end{equation*}
Combining the two inequalities proves the proposition.
\end{proof}

\begin{rmk}
The analogous proposition holds for the emergence of a continuous self-map $f\colon X\to X$.
\end{rmk}	
						
\section{The Epstein--Vogt model}  \label{section: Epstein-Vogt's counterexample}

In a breakthrough paper from 1978 (see \cite{EV78}), David Epstein and Elmar Vogt presented the first codimension-three counterexample to the Periodic Orbit Conjecture: they construct a polynomial vector field on $\R^{7}$ whose orbits are all periodic and the periods are unbounded, that lives in a $4$-dimensional submanifold also defined by polynomial equations (in particular, both the vector field and the manifold are analytic). 
							
							In the present section, we recall in some detail their construction while generalizing it. The proofs are on the same lines as those in \cite{EV78}.
							\subsection{Construction of the manifold}
							
							Consider the following domain of $\R^{2}$: \begin{equation}\label{octogonal domain} D=\{(x,y)\in \R^{2} \mid -2\leq x\leq 2, -2\leq y\leq 2, -3\leq x+y\leq 3, -3\leq x-y\leq 3\}.\end{equation} Define the function $\psi_{0}\colon \R^{2}\to \R$ to be \begin{align}\label{psi_0}
								\psi_{0}(x,y)&=(2-x)(2+x)(2-y)(2+y)(3+x+y)(3-x-y)(3+x-y)(3-x+y)\nonumber \\&=(4-x^2)(4-y^2)(9-(x+y)^2)(9-(x-y)^2) .\end{align} Notice that $\psi_{0}$ is positive on the interior of $D$ and is zero on the boundary $\partial D$ of $D$. Moreover, by Lemma 3.2 in \cite{EV78}, $d\psi_{0}$ does not vanish on $\interior{D}-\{0\}$.
							
							Now, let $\psi=\psi_{0}\cdot \varphi$, where $\varphi\colon \R^{2}\to \R$ satisfies: \begin{enumerate}[label=\((\roman*)\)]
								\item \label{phi eh c infinito} $\varphi$ is $C^{\infty}$;
								\item \label{phi eh positiva} $\varphi\geq 0$ on $D$ and $\psi>0$ on $\interior{D}$;
								\item \label{psi é submersao} the differential of $\psi=\psi_{0}\cdot \varphi$ does not vanish on $\interior{D}-\{0\}$;
								\item \label{phi eh simetrica} $\varphi$ satisfies that for every $(x,y)\in \R^{2}$, $$\varphi(x,y)=\varphi(-x,y)=\varphi(x,-y).$$
							\end{enumerate} In particular, condition \ref{phi eh simetrica} implies $\varphi(x,y)=\varphi(-x,-y)$. The original Epstein--Vogt example is obtained by taking $\varphi\equiv1$.

							Multiplying $\psi$ by a positive constant if necessary, we may assume $\max{\{\psi(x,y) \mid (x,y)\in D\}}>1$. Define the set $A$: \begin{equation}\label{definicao do anel} A=D\cap \{(x,y)\in \R^{2} \mid  \psi(x,y)\leq 1\}.\end{equation} We claim that $A$ is a topological annulus, as the one depicted in Figure \ref{fig: octogono_original_cinza}. 
							
							\begin{lemma} \label{lemma: A is annulus}
								The set $A$ is homeomorphic to $\T \times [0,1]$.
							\end{lemma}	
							\begin{proof}
We recall that $\psi(A)=[0,1]$ and that the boundary of $A$ is formed by the level set $\{\psi=0\}$ which is an octagon and $\{\psi=1\}$ which is a circle. 
								As $\nabla \psi ^2= 2\psi \nabla \psi$, by property \ref{psi é submersao}, the differential $\nabla \psi^2$  does not vanish on $D\setminus \{0\}\supset \interior A$. 
								
								Now take $v^2\in  (0,1)$. The level set $\psi^{-1}(v)= (\psi^2)^{-1}(v^2)$ is an embedded  curve in $\interior A$. It is also compact because it is a closed subset of $(\psi)^{-1}([v/2, (1+v)/2 ])\subset A$. Hence it is a disjoint union of circles. 
								Inside each of these circles, $\restr{\psi}{D}$ displays a maximum and so a critical point. By \ref{psi é submersao}, there is a unique critical point and so the level set is a unique circle. Thus on $A\cap \{\psi >0\}$,  each point belongs to a level set of  $\psi^2$ diffeomorphic to a circle, which is its orbit as the gradient of $\psi^2$ does not vanish on it. This proves the lemma.
							\end{proof}

							\begin{rmk}\label{rmk: level sets are circles}
								The proof of Lemma \ref{lemma: A is annulus} establishes that each level set of the Hamiltonian $H=\psi^{2}/2$ is a circle. Hence, every orbit of the associated Hamiltonian flow will be a circle.
							\end{rmk}

							\begin{figure}[ht]
								\centering
								\IfFileExists{octogono_original_cinza.jpeg}{%
                                \includegraphics[width=0.325\textwidth]{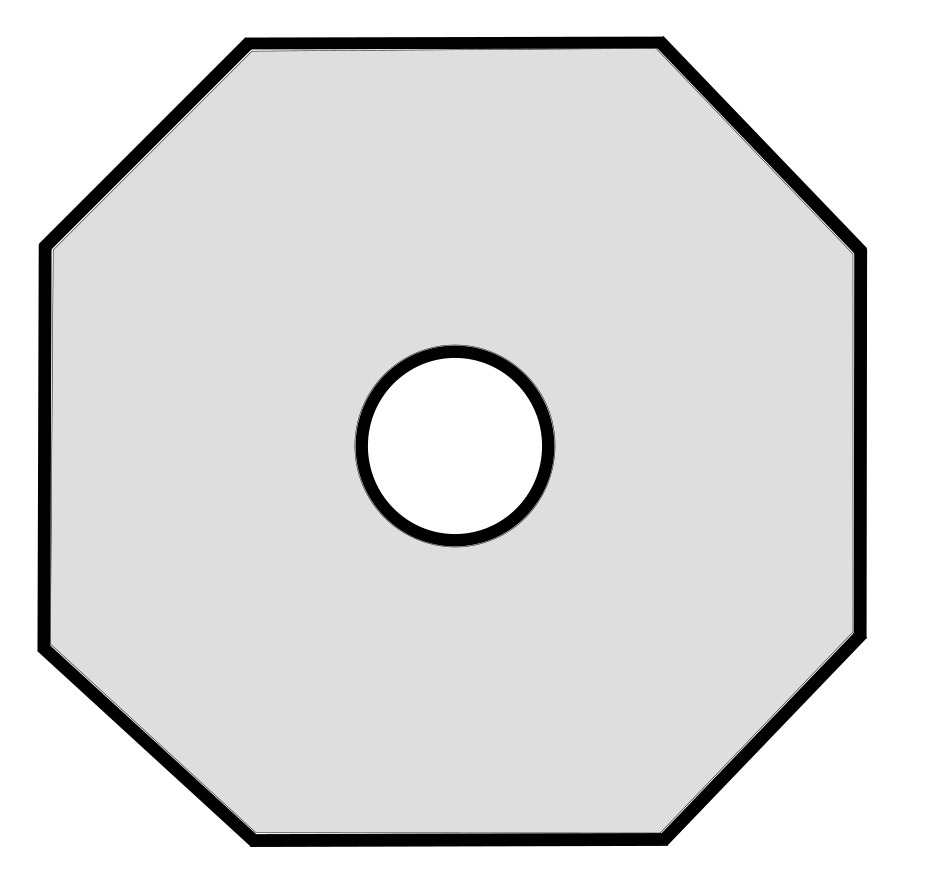}%
                            }{%
                                \fbox{\parbox{0.325\textwidth}{\centering Missing figure:\\
                                \texttt{octogono_original_cinza.jpeg}}}%
                            }
								\caption{The annulus $A$ for $\psi=\psi_{0}$.}
								\label{fig: octogono_original_cinza}
							\end{figure}
							
							Next, define a function $\rho\colon \R^{2}\to \R$ to be $$\rho(x,y)\coloneqq (1-\psi(x,y)) (9-(x+y)^{2})(9-(x-y)^{2}),$$ and a map $F\colon \R^{7}\to \R^{3}$ by $F(\xi)=(F_{1}(\xi),F_{2}(\xi),F_{3}(\xi))$ where, for each $\xi=(x,y,u_{1},u_{2},w_{1},w_{2},z)\in \R^{7}$, we set
							$$\begin{cases}
								F_{1}(\xi)=u_{1}^{2}+u_{2}^{2}-4+x^{2},\\
								F_{2}(\xi)=w_{1}^{2}+w_{2}^{2}-4+y^{2},\\
								F_{3}(\xi)=z^{2}-\rho(x,y).
							\end{cases}$$ Finally, we set $M=F^{-1}(0)$. 
							
							\begin{lemma} \label{lemma: M see projeta sobre A versao psi}
								The projection of $\R^{7}$ onto the first two coordinates maps $M$ onto $A$.
							\end{lemma}
							\begin{proof}
								The present proof is similar to Lemma 4.1 on \cite{EV78}. Let $\xi=(x,y,u_{1},u_{2},w_{1},w_{2},z)\in M$. Since $F_{1}(\xi)=F_{2}(\xi)=F_{3}(\xi)=0$, obtain $$(\ast)\begin{cases}
									0&\leq u_{1}^{2}+u_{2}^{2}=4-x^{2},\\
									0&\leq w_{1}^{2}+w_{2}^{2}=4-y^{2},\\
									0&\leq z^{2}=\rho(x,y),
								\end{cases}$$ so that $-2\leq x\leq 2$ and $-2\leq y\leq 2$. Moreover, \begin{align*}
									0\leq \rho(x,y)(4-x^{2})(4-y^{2})=(1-\psi(x,y))\cdot \psi_{0}(x,y).
								\end{align*} Since \(\psi=\psi_0\cdot \varphi\) and \(\varphi\ge0\) on \(D\) by  \ref{phi eh positiva}, we obtain
								\[
								0\le (1-\psi(x,y))\psi_0(x,y)\varphi(x,y)=(1-\psi(x,y))\psi(x,y).
								\] Hence $0\leq \psi(x,y)\leq 1$.

								Suppose $(x,y)\notin A$. Because $\psi_{0}$ is invariant by the symmetries $x\mapsto -x$ and $y\mapsto -y$, and since we required the same for $\varphi$ in \ref{phi eh simetrica}, the product \(\psi=\psi_0\cdot \varphi\) will also be invariant. Hence, without loss of generality we may assume \(x,y\ge0\).
								
								Then, by $(\ast)$, we have $x+y>3$, together with $0\leq x\leq 2$ and $0\leq y\leq 2$. This implies $1\leq x\leq 2$ and $1\leq y\leq 2$, and consequently, $-1\leq x-y\leq 1$.
								
								Since $\rho(x,y)\geq 0$, $(9-(x+y)^{2})<0$, and $(9-(x-y)^{2})>0$, we have: 
								$$1-\psi(x,y)\leq 0,$$ i.e., $1\leq \psi(x,y)$. Combining this with the previous bound \(\psi(x,y)\le1\) yields \(\psi(x,y)=1\).
								
								On the other hand, we claim that if $\xi\in M$ and $x+y>3$, then $\psi(x,y)\leq 0$. Indeed, $\varphi(x,y)\geq 0$ by hypothesis \ref{phi eh positiva}; $(\ast)$ implies $(4-x^{2})\geq 0$ and $(4-y^{2})\geq 0$; finally, as we just saw, $(9-(x+y)^{2})<0$ and $(9-(x-y)^{2})>0$, so: $$\psi(x,y)=\psi_{0}(x,y)\varphi(x,y)=(4-x^{2})(4-y^{2})(9-(x+y)^{2})(9-(x-y)^{2})\cdot \varphi(x,y)\leq 0,$$ i.e., we just proved that if $(x,y)\notin A$ then $1= \psi(x,y)\leq 0$, a contradiction. This proves that $\proj(M)\subset A$.
								
								Conversely, let $(x,y)\in A$. To show the existence of $\xi\in M$ which projects to $(x,y)$ it suffices to define $z=\sqrt{\rho(x,y)}, u_{1}=\sqrt{4-x^{2}},u_{2}=0,w_{1}=\sqrt{4-y^{2}},$ and $w_{2}=0$, proving that $\proj(M)=A$.
							\end{proof}
							
							We now study the regularity of the algebraic variety $M$.
							
							\begin{lemma} \label{lemma: M eh variedade pra psi}
								The map $F\colon \R^{7}\to \R^{3}$ is regular on $M=F^{-1}(0)$. Consequently, \(M\) is a smooth \(4\)-dimensional submanifold of \(\R^7\).
							\end{lemma}\begin{proof}
								At a point $\xi=(x,y,u_{1},u_{2},w_{1},w_{2},z)\in M$, the Jacobian of $F$ is $$dF_{\xi}=\begin{bmatrix}
									2x&0&2u_{1}&2u_{2}&0&0&0\\
									0&2y&0&0&2w_{1}&2w_{2}&0\\
									-\partial \rho/\partial x&-\partial \rho/\partial y&0&0&0&0&2z
								\end{bmatrix}.$$ We will prove that $F$ is a submersion by showing that its Jacobian has rank $3$.
								
								We start by observing that neither of the two first rows can be zero on $M$: for example, $u_{1}^{2}+u_{2}^{2}=4-x^{2}$ implies that $(x,u_{1},u_{2})\neq 0$ (and similarly, $w_{1}^{2}+w_{2}^{2}=4-y^{2}$ implies that $(y,w_{1},w_{2})\neq 0$). So the Jacobian has rank 2 or 3. Let's prove it is 3. 
								
								For the sake of contradiction, assume that the third row is not independent of the first two. Then $z=0$ and so $\rho(x,y)=0$. We have two cases to deal with:
								
								\begin{itemize}   
									\item \textit{\textbf{Case 1}}: $(9-(x+y)^{2})(9-(x-y)^{2})\neq 0$. Then $\psi(x,y)=1$ and so 
									$$d\rho=-d\psi\cdot (9-(x+y)^{2})(9-(x-y)^{2}),$$
									which is non-zero by hypothesis \ref{psi é submersao}. Also, since $\psi(x,y)=1$, $4-y^{2}\neq0\neq 4-x^{2}$ and so both $u_{1}^{2}+u_{2}^{2}$ and $w_{1}^{2}+w_{2}^{2}$ must be non-zero. This implies that the first two rows are independent of the third - a contradiction.

									\item \textit{\textbf{Case 2}}: $(9-(x+y)^{2})(9-(x-y)^{2})=0.$ Hence, $\psi(x,y)=0$ (since $\psi_{0}(x,y)=0$). Applying the symmetries $x\mapsto -x$ and $y\mapsto -y$ if necessary\footnote{We recall that by hypothesis \ref{phi eh simetrica} we may assume it for $\psi$.}, we may assume that $x+y=3$. Moreover,
									$$d\rho=(1-\psi(x,y))\cdot (-2(x+y)(dx+dy))\cdot (9-(x-y)^{2}),$$
									and since $\psi(x,y)=0$ and $x+y=3$, we have $$d\rho=-6(dx+dy)(9-(x-y)^{2})=6(dx+dy)((x-y)^{2}-9).$$ Also, since $(x,y)\in A$, the fact that $x+y=3$ implies that $(x-y)^{2}-9\neq 0$. So, $\frac{\partial \rho}{\partial x}=\frac{\partial \rho}{\partial y}\neq 0$. Then the third row is independent of the first two if $u_{1}=u_{2}=w_{1}=w_{2}=0$, i.e., $4-x^{2}=0$ and $4-y^{2}=0$, which cannot occur simultaneously for $(x,y)\in A$. This proves the rank must be 3.
								\end{itemize}
							\end{proof}
							
							\begin{rmk}\label{rmk: variedades se intersectam transversalmente}
								We have actually proved that $M$ is the transverse intersection of the following hypersurfaces: 
								
								\[S_{1}\coloneqq \{(x,y,u_{1},u_{2},w_{1},w_{2},z)\in \R^{7}\mid u_{1}^{2}+u_{2}^{2}+x^{2}=4\}, \]
								\[S_{2}\coloneqq	\{(x,y,u_{1},u_{2},w_{1},w_{2},z)\in \R^{7}\mid w_{1}^{2}+w_{2}^{2}+y^{2}=4\},\]
								and 
								$$S_{3}\coloneqq\{(x,y,u_{1},u_{2},w_{1},w_{2},z) \in \R^{7}\mid z^{2}=\rho(x,y) \text{ and } (x,y) \notin \{(0,3),(0,-3),(3,0),(-3,0)\} \} .$$
							\end{rmk}
							
							\begin{lemma}\label{lemma: M e compact pra phi}
								The $4$-dimensional manifold $M$ is compact.
							\end{lemma}
							\begin{proof} $M$ is closed because it is the preimage of the point $0\in\R^3$ under $F$ and bounded since its projection onto the first two coordinates is $A$ and the other variables satisfy $(\ast)$. 
							\end{proof}
							
							\subsection{Construction of the vector field}
							
							We will find the equations on each of the seven variables of $M$. Over $A$ we take the Hamiltonian flow with $H=\frac{\psi^{2}}{2}$. Therefore the associated Hamiltonian system on the \((x,y)\)-plane is

							\begin{equation} \label{def: eqs de Hamilton no anel}
								\begin{cases}
									\dot{x}=\psi\cdot \frac{\partial \psi}{\partial y}\\
									\dot{y}=-\psi\cdot \frac{\partial \psi}{\partial x}
								\end{cases}.
							\end{equation} 
							Hence $\psi$ is constant along the orbits and $(\dot{x},\dot{y})$ is zero on $M$ if and only if $\psi=0$. This guarantees that the orbits in the $(x,y)$-plane are simple closed curves in $A$, except when $\psi=0$. To determine $\dot{z}$ differentiate $z^{2}=\rho(x,y)$: 
							$$2z\dot{z}=\frac{d}{dt}\rho(x,y).$$
							Differentiating the logarithm of $\rho(x,y)=(1-\psi(x,y))(9-(x+y)^{2})(9-(x-y)^{2})$, we see that the right-hand side of the above equality becomes $$2\rho(x,y)\sigma(x,y),$$ where \begin{equation}\label{def: sigma}\sigma(x,y)\coloneqq \left(\frac{\partial \psi}{\partial x}-\frac{\partial \psi}{\partial y}\right)\cdot \frac{(x+y)\cdot \psi(x,y)}{9-(x+y)^{2}}+\left(\frac{\partial \psi}{\partial x}+\frac{\partial \psi}{\partial y}\right)\cdot \frac{(y-x)\cdot \psi(x,y)}{9-(x-y)^{2}}.\end{equation}
							
							Indeed, \begin{align*}
								\frac{d}{dt} \log{\rho(x,y)}&= \frac{d}{dt} \left(\log{(1-\psi)}+\log{(9-(x+y)^{2})}+\log{(9-(x-y)^{2})}\right)\\&=\frac{1}{1-\psi}\cdot \left(-\dot{\psi}\right)-\frac{2(x+y)(\dot{x}+\dot{y})}{9-(x+y)^{2}}-\frac{2(x-y)(\dot{x}-\dot{y})}{9-(x-y)^{2}}\\&=\frac{2(x+y)\psi}{9-(x+y)^{2}}\cdot \left( \frac{\partial \psi}{\partial x}- \frac{\partial \psi}{\partial y}\right)+\frac{2(y-x)\psi}{9-(x-y)^{2}}\cdot \left( \frac{\partial \psi}{\partial x}+\frac{\partial \psi}{\partial y}\right),
							\end{align*} so that \begin{equation}\label{eq: sigma log de rho} \frac{d}{dt}\log{\rho(x,y)}=\frac{1}{\rho(x,y)}\cdot \frac{d}{dt}\rho(x,y)=2\cdot \sigma(x,y),\end{equation} and then $$\frac{d}{dt}\rho(x,y)=2\rho(x,y)\sigma(x,y).$$ Together with $z^{2}=\rho(x,y)$, this leads us to define: 
							
							\begin{equation}\label{def:dot-z}
								\dot{z}=z\cdot \sigma(x,y).
							\end{equation}
							
							We now define the dynamics in the \(u\)- and \(w\)-planes. Begin with the $u$-plane: we would like to have a circular motion about the center. Such a flow is given by $$\begin{cases}
								\dot{u}_{1}=-pu_{2}\\
								\dot{u}_{2}=pu_{1}
							\end{cases},$$ where $p=\dot{\theta}$ is the angular velocity. 
							
							However, the radius of the circle in the $u$-plane must simultaneously change with time, since $$\frac{d}{dt}(u_{1}^{2}+u_{2}^{2})=-2x\dot{x}=-2x\psi\cdot \frac{\partial \psi}{\partial y}.$$ Then we should have a set of equations of the form: $$\begin{cases}
								\dot{u}_{1}=Ku_{1}-pu_{2}\\
								\dot{u}_{2}=pu_{1}+Ku_{2}
							\end{cases},$$ where $p$ and $K$ are functions on $(x,y)$. So if we want $u_{1}$ and $u_{2}$ to satisfy these equations, we must have: \begin{align*}
								\frac{d}{dt}(u_{1}^{2}+u_{2}^{2})&=2u_{1}\dot{u}_{1}+2u_{2}\dot{u}_{2}\\
								&=2u_{1}(Ku_{1}-pu_{2})+2u_{2}(pu_{1}+Ku_{2})\\&=2K(u_{1}^{2}+u_{2}^{2})\\&=2K(4-x^{2}), 
							\end{align*} and hence $$2K(4-x^{2})=-2x\psi\cdot \frac{\partial \psi}{\partial y}.$$ Since $\psi(x,y)=\varphi\cdot (4-x^{2})(4-y^{2})(9-(x+y)^{2})(9-(x-y)^{2})$, we set $$K(x,y)=-x\cdot \varphi \cdot \left(\frac{\partial \psi}{\partial y}\right)\cdot (4-y^{2})(9-(x+y)^{2})(9-(x-y)^{2}).$$ Similarly, $w_{1},w_{2}$ must satisfy $$\begin{cases}
								\dot{w}_{1}=Lw_{1}-qw_{2}\\
								\dot{w}_{2}=qw_{1}+Lw_{2}
							\end{cases},$$ where $q$ and $L$ are functions on $(x,y)$, and so we set $$L(x,y)=y\cdot \varphi \cdot \left(\frac{\partial \psi}{\partial x}\right)\cdot (4-x^{2})(9-(x+y)^{2})(9-(x-y)^{2}).$$ 
							
							We now choose the angular speeds \(p,q:\R^2\to\R\).  We require:
						   \begin{enumerate}[label=\((\roman*)\)]
								\item \label{p e positivo} $p(x,y)>0$ if $0<y\leq 2$;
								\item \label{q e positivo} $q(x,y)>0$ if $0<x\leq 2$;
								\item \label{p e simetrico} $p(x,y)=p(-x,y)=-p(x,-y)$;
								\item \label{q e simetrico} $q(x,y)=-q(-x,y)=q(x,-y)$;
								\item \label{p=q} $p(x,y)=q(x,y)$ at $x+y=3$.        
							\end{enumerate} For example, we can set $p$ and $q$ as: $p(x,y)=(9+x^{2}-y^{2})y$, and $q(x,y)=(9-x^{2}+y^{2})x$.

							Define the vector field $X$ by: \begin{align*}X(\xi)=\psi \left(\frac{\partial\psi}{\partial y}\right)\frac{\partial}{\partial x}-\psi \left(\frac{\partial\psi}{\partial x}\right)\frac{\partial}{\partial y}+&(Ku_{1}-pu_{2})\frac{\partial}{\partial u_{1}}+(Ku_{2}+pu_{1})\frac{\partial}{\partial u_{2}}+\\+&(Lw_{1}-qw_{2})\frac{\partial}{\partial w_{1}}+(Lw_{2}+qw_{1})\frac{\partial}{\partial w_{2}}+z\sigma(x,y)\frac{\partial}{\partial z},\end{align*} for each $\xi=(x,y,u_{1},u_{2},w_{1},w_{2},z)\in \R^{7}$. By construction, we have: 
							
							\begin{lemma} \label{lemma: X e tangent a M} 
								If $\xi \in M$, then $X(\xi)$ is tangent to $M$.
							\end{lemma}
							\begin{proof}
							Since $F$ is regular on $M$ it suffices to prove that, for every $\xi \in M$, $dF_{\xi}(X(\xi))=0$, which is a routine calculation.
							\end{proof}

							\subsubsection{Properties of $X$}
							
							In this subsection we study properties of the vector field \(X\). Our goal is to show that \(X\) is a smooth, nowhere-vanishing vector field tangent to \(M\), such that every orbit is diffeomorphic to a circle, and that the periods can be arbitrarily large.

							\begin{lemma} \label{lemma: X nunca e zero em M}
								The vector field $X$ does not vanish on $M$.
							\end{lemma}
							\begin{proof}  
								Suppose $\dot{u}_{1}=\dot{u}_{2}=0$. Then  
								\begin{align*}
									0&=u_{1}(Ku_{2}+p(x,y)\cdot u_{1})-u_{2}(Ku_{1}-p(x,y)\cdot u_{2})
									\\&=p(x,y)\cdot (u_{1}^{2}+u_{2}^{2}) = p(x,y)\cdot (4-x^2)\; .
								\end{align*}
								As $p=0$ on $A$ only at $\{y=0\}$, this implies that $(x,y)$ belongs to $\{y=0\}\cup \{ x= \pm2\}$.  Likewise, if $\dot{w}_{1}=\dot{w}_{2}=0$, the point $(x,y)$ belongs to $\{x=0\}\cup \{ y= \pm2\}$.  As the intersection
								\[ ( \{ x= \pm2\}\cup \{y=0\})\cap (\{x=0\}\cup \{ y= \pm2\})\]
								does not meet $A$, by Lemma  \ref{lemma: M see projeta sobre A versao psi}, the vector field $X$ never vanishes.
							\end{proof}
							
							To prove that the orbits of the flow are circles, we will prove, in two lemmas, that the vector field $X$ respects several symmetries.
							
							\begin{lemma} \label{lemma: invariancia por rotacao em u e em w}
								The functions \(\psi,\rho,\sigma\), the map \(F\), and the vector field \(X\) are invariant under rigid rotations in the \(u\)-plane and under rigid rotations in the \(w\)-plane.
					    	\end{lemma}
							\begin{proof}
								Since $\psi$, $\rho$, and $\sigma$ do not depend on $u_{1},u_{2},w_{1}$, and $w_{2}$, and since $F$ only depends on the norm of $(u_{1},u_{2})$ and $(w_{1},w_{2})$, it is clear that they remain unchanged by such rotations. To see that the vector field $X$ is also invariant, observe that the matrix $\begin{bmatrix}
									K&-p\\p&K
								\end{bmatrix}$ is a real multiple of a rotation since it is 0 or equal to $$\sqrt{K^{2}+p^{2}}\cdot \begin{bmatrix}
									\frac{K}{\sqrt{K^{2}+p^{2}}}&\frac{-p}{\sqrt{K^{2}+p^{2}}}\\\frac{p}{\sqrt{K^{2}+p^{2}}}&\frac{K}{\sqrt{K^{2}+p^{2}}}
								\end{bmatrix},$$ that depends only on $x$ and $y$. Similarly, we have that $\begin{bmatrix}
									L&-q\\q&L
								\end{bmatrix}$ is also a scalar multiple of a rotation. These matrices commute with the rigid rotations. The lemma follows.
							\end{proof}
							\begin{lemma} \label{lemma: invariancia pela involucao T pra psi}
								Let $T:\R^{7}\mapsto \R^{7}$ be the involution defined by $$T(x,y,u_{1},u_{2},w_{1},w_{2},z)=(-x,-y,u_{2},u_{1},w_{2},w_{1},z).$$
								
								Then $\psi$, $\rho$, $F$, $M$, and $X$ are invariant under $T$.
							\end{lemma}
							\begin{proof}
								We begin by observing that $\psi$, $\rho$, and $F$ are invariant under the symmetries $x\mapsto -x$ and $y\mapsto -y$. Since $\psi$ and $\rho$ do not depend on $u_{1},u_{2},w_{1}, w_{2}$, and since $F$ is invariant under $(u_{1},u_{2})\mapsto  (u_{2},u_{1})$ and $(w_{1},w_{2})\mapsto  (w_{2},w_{1})$, we conclude that they are  also $T$-invariant. Hence, for every $\xi \in M$, if $F(\xi)=0$, $F(T(\xi))=0$, proving that $M$ is also left invariant by $T$.

								Finally, we need to verify the analogous for $X$. First recall that that the $xy$-coordinate of $X$ is given by the Hamiltonian flow of $(x,y)\mapsto \frac{1}{2}\psi(x,y)^2$. As $\psi(-x,-y)=\psi(x,y)$, we obtain that the $xy$-coordinates of $X$ and $T_* X$ are the same. 
								Also $xyz$-coordinate of $X$ is the lifting of the $xy$-coordinates of $X$ to the surface $\{ z^2= \rho(x,y)\}$. As $\rho(-x,-y)= \rho(x,y)$, the $xyz$-coordinates of $X$ and $T_* X$ are the same.
								
								As $\psi(-x,y)= \psi(x,y)$, it holds $\partial_x\psi(-x,y) =-\partial_x\psi(x,y)$. Likewise $\psi(x,-y)= \psi(x,y)$ implies $\partial_y\psi(x,-y) =-\partial_y\psi(x,y)$. Hence $K(-x,-y)= K(x,y)$ and  $L(-x,-y)= L(x,y)$.
								
								 Next, by the hypothesis \ref{p e simetrico} and \ref{q e simetrico} on the functions $p$ and $q$, we conclude that $p$ and $q$ are anti-invariant. Hence, $$T_{\ast}\left((Ku_{1}-pu_{2})\frac{\partial}{\partial u_{1}}\right)=(Ku_{2}+pu_{1})\frac{\partial}{\partial u_{2}},$$ and $$T_{\ast}\left((pu_{1}+Ku_{2})\frac{\partial}{\partial u_{2}}\right)=(-pu_{2}+Ku_{1})\frac{\partial}{\partial u_{1}},$$ i.e.,$$T_{\ast}\left((Ku_{1}-pu_{2})\frac{\partial}{\partial u_{1}}+(pu_{1}+Ku_{2})\frac{\partial}{\partial u_{2}}\right)=(Ku_{1}-pu_{2})\frac{\partial}{\partial u_{1}}+(pu_{1}+Ku_{2})\frac{\partial}{\partial u_{2}}.$$ This proves $T_{\ast}X=X$, i.e., that $X$ is $T$-invariant. 
								\end{proof}
								
								\begin{prop} \label{prop: psi>0 orbitas circulos pra psi}
								On $M$, if $\psi>0$, then each orbit is (diffeomorphic to) a circle. As $\psi$ tends to zero, the period tends to infinity.
								\end{prop}
								
								To prove Proposition \ref{prop: psi>0 orbitas circulos pra psi}, we need the following  lemma:
								\begin{lemma} \label{lemma: psi>0 (x,y) circulos}
									On $A\cap \{\psi >0\}$,  each orbit $\gamma$ of the Hamiltonian $\psi^2$ is (diffeomorphic to) a circle. Moreover, if the period of $\gamma$ is $2\lambda>0$, then $\gamma(\lambda)=-\gamma(0)$. Finally, 
									as $\psi$ tends to zero, the time-of-first-return tends to infinity.
								\end{lemma}
								\begin{proof}
									By Remark \ref{rmk: level sets are circles}, each orbit of the Hamiltonian $\psi^{2}/2$ is a circle. Let $\gamma(t)=(x(t),y(t))$ be an orbit in $A\cap \{\psi >0\}$ and let $2\lambda$ be its period. 
									
									Now, we claim that $\gamma(\lambda)=-\gamma(0)$. To see this, we begin by the observation that $ -\gamma $ is also an orbit since the Hamiltonian satisfies $\psi^2(x,y)=\psi^2(-x,-y)$, and so 
									\[ -\gamma = -J (\nabla \psi^2 )(\gamma)= J (\nabla \psi^2 )(-\gamma)\;  . \]
									
									Moreover, since both $ \gamma $ and $-\gamma $ lie in the same level curve of $\psi$, there must exist $\tau>0$ (which we can suppose minimal) with $- \gamma (\tau)=\gamma(0)$.
									
									Define $\widehat{\gamma}(t)=-\gamma(t+\tau)$. This is also an orbit of the Hamiltonian, which starts at $\widehat{\gamma}(0)=-\gamma(\tau)= \gamma(0)$.  So, $\widehat{\gamma}(t)=\gamma(t)$ for all $t\in \R$, i.e., 
									$$-\gamma(t+\tau)=\gamma(t)\text{ for all $t$.}$$
									In particular, for $t=\tau$, we have: $-\gamma(2\tau)=\gamma(\tau).$ On the other hand, since $\gamma(\tau)=-\gamma(0) $, we conclude that $$\gamma(2\tau)=\gamma(0),$$ i.e., $2k\cdot \tau=2\lambda$ for some $k\in \N$. If $k$ is odd, we are done, since $\gamma(\lambda)=\gamma((2k^{\prime} +1)\tau)=\gamma(\tau)=-\gamma(0)$.  If  $k=2k^{\prime}$ is even, then   $\gamma(2k^\prime \tau)=\gamma(\lambda)=\gamma(0)$. This would imply that $2\lambda$ is not the minimal period for $\gamma$, a contradiction. This proves $\gamma(\lambda)=-\gamma(0)$, i.e., $x(\lambda)=-x(0)$ and $y(\lambda)=-y(0)$.

									Finally, as every orbit $\gamma$ of $A\cap \{\psi >0\}$ is an essential circle, there exists $m>0$ such that their length is bounded from below by $m$. On the other hand, we recall that $\psi^2$ is constant on each orbit and so:
									\begin{align*}
										\ell(\gamma)&=\psi(\gamma(0))\cdot\int_{0}^{2\lambda} \|\nabla \psi (\gamma(t)) \| \ dt.
									\end{align*} 
									As $ \ell(\gamma)$ is bounded from below and $\|\nabla \psi\|$ is bounded from above, when $\psi(\gamma(0))\to 0$ we must have   $\lambda \to +\infty$.
									\end{proof}

									\begin{proof}[Proof of Proposition \ref{prop: psi>0 orbitas circulos pra psi}]
									Let $\xi(t)=(x(t),y(t),u_{1}(t),u_{2}(t),w_{1}(t),w_{2}(t),z(t))\in M$ be an orbit of $X$. Since $\psi>0$, $(\dot{x},\dot{y})\neq 0$, and since, by Lemma \ref{lemma: psi>0 (x,y) circulos}, the level curves in $A$ are simple closed curves. Let $2\lambda >0$ be the period of $(x(t),y(t))$. We will show that $\xi(2\lambda)=\xi(0)$.
									
									By  Lemma \ref{lemma: psi>0 (x,y) circulos} we have that $x(\lambda)=-x(0)$ and $y(\lambda)=-y(0)$. 
									
									If $z(0)=0$ then $\rho(x(0),y(0))=0$ and, since $\psi>0$ (so $(x(0)-y(0))^{2}\neq 9$ and $(x(0)+y(0))^{2}\neq 9$), we must have $\psi(x(0),y(0))=1$. Since the level sets are invariant, we conclude that $\psi(x(t),y(t))=1$ for all $t\in \R$, so that $z(t)=0$ for every $t$. This also implies that if $z(0)>0$ then $z(t)>0$ for all $t$, and if $z(0)<0$ then $z(t)<0$, for all $t\in \R$. Moreover, by definition of $z$, $z^{2}=\rho(x,y)$. As $\rho(-x,-y)=\rho(x,y)$, we have  $z(\lambda)=z(0)=z(2\lambda)$.
									
									Finally we need to show the coordinates $u_{1},u_{2},w_{1},w_{2}$ have the same period. By definition of $F_{1}$ and $F_{2}$ we have: $$u_{1}^{2}(\lambda)+u_{2}^{2}(\lambda)=u_{1}^{2}(0)+u_{2}^{2}(0),$$ and $$w_{1}^{2}(\lambda)+w_{2}^{2}(\lambda)=w_{1}^{2}(0)+w_{2}^{2}(0).$$ Therefore, there exist rotations $R$ and $S$, in the $u$-plane and in the $w$-plane, respectively, such that $$\xi(\lambda)=RST\xi(0).$$ Of course, $R$ and $S$ depend on $\xi(0)$. Since $R_{\ast}X=S_{\ast}X=T_{\ast}X=X$ by Lemma \ref{lemma: invariancia por rotacao em u e em w}, we then conclude, by the uniqueness of solutions of ODEs, that: $$\xi(t+\lambda)=RST\xi(t),$$ for all $t\in \R$.
									
									Now we observe that $T$ is involutive and conjugates $R$ and $S$ with their inverse, i.e., $TRT=T^{-1}RT=R^{-1}$ and $TST=T^{-1}ST=S^{-1}$, so that $RSTRST=I$, where $I$ is the identity. Hence, $$\xi(2\lambda)=RST\xi(\lambda)=RSTRST\xi(0)=\xi(0).$$ This shows the orbits are simple closed curves of period $2\lambda>0$.
									
									To see that the period of $\xi$ goes to infinity, observe that we already know, by Lemma \ref{lemma: psi>0 (x,y) circulos}, that the period $2\lambda$ of the orbit $(x(t),y(t))$ on $A$ goes to infinity as the initial point of the orbit goes to the boundary of $D$. Since the period of $\xi$ is the same, we are done.
									\end{proof}

									\begin{rmk}\label{rmk: period is given by the flow on A}
									Notice that the period of the flow is then given by the period of its projection onto $A$, i.e., the Hamiltonian flow on $A$ with $H=\psi^{2}/2$.
									\end{rmk}	
									
									\begin{prop} \label{prop: orbitas circulos em x+y=3 pra psi}
									If $x+y=\pm 3$ or $x-y=\pm 3$, then the orbit through \(\xi\) is a circle.
									\end{prop}
									\begin{proof}
									Since $x+y=\pm 3$ or $x-y=\pm 3$, we have that $\psi(x,y)=0$ (since $\psi_{0}(x,y)=0$) and that $\rho(x,y)=0$. Hence, $z(0)=0$ and so $z(t)=0$ for all $t$. Also, $\psi(x,y)=0$ guarantees $\dot{x}=\dot{y}=0$, i.e., $x$ and $y$ are independent of $t$. Moreover, by definition of both $K$ and $L$, $K(x(0),y(0))=L(x(0),y(0))=0$. Because the equation $x(t)+y(t)=\pm 3$ or $x(t)-y(t)=\pm 3$ will be satisfied for all $t$, since $x$ and $y$ do not depend on $t$, we actually have $K(x(t),y(t))=L(x(t),y(t))=0$, for all $t\in \R$. Hence, throughout the orbit, the following equations are satisfied: $$\begin{cases}
										\dot{x}=\dot{y}=\dot{z}=0\\
										\dot{u}_{1}=-pu_{2}\\
										\dot{u}_{2}=pu_{1}\\
										\dot{w}_{1}=-qw_{2}\\
										\dot{w}_{2}=qw_{1}
									\end{cases}.$$ Since we have chosen $p$ and $q$ to satisfy property \ref{p=q} above, we must have $p=q$ or $p=-q$, so we conclude that $|p|=|q|$. This shows that the angular speed in both the $u$-plane and the $w$-plane is $|p|=|q|$, which is non-zero since $p=q=0$ if and only if $x=y=0$, and $(x,y)\notin A$ (and $M$ projects onto $A$). This proves that the orbit through such a point is a circle with time of first return equal to $\frac{2\pi}{|p|}$. 
									\end{proof}
									
									\begin{prop} \label{lemma: orbitas circulos em x=2 pra psi}
									If $x=\pm 2$ or $y=\pm 2$, then the orbit through \(\xi\) is a circle.
									\end{prop}
									\begin{proof}
									Suppose, for example, that $x=2$. Then by $u_{1}^{2}+u_{2}^{2}=4-x^{2}$, we have $u_{1}=u_{2}=0$. Also, since $\psi=0$ (because $\psi_{0}=0$), we have $\dot{x}=\dot{y}=0$: this shows that $(x(t),y(t))$ is constant and, hence $z(t)$ also is (it depends only on $x$ and $y$), implying $\dot{z}=0$. We then have: $$\dot{x}=\dot{y}=\dot{u}_{1}=\dot{u}_{2}=\dot{z}=0.$$ Moreover, $L(x,y)=0$ through the orbit. Indeed, $\dot{x}=0$ implies that $x(t)=2$ for all $t$, and so $L(x(t),y(t))=0$ for all times $t$. So, on the orbit, we have the equation: $$\begin{cases}
										\dot{w}_{1}=-qw_{2}\\
										\dot{w}_{2}=qw_{1}
									\end{cases},$$ where $w^{2}_{1}+w^{2}_{2}=4-y^{2}\neq 0$. The orbit is then a circle with time of first return $\frac{2\pi}{|q|}$. 
									\end{proof}
									
									\begin{fact}\label{prop: orbitas vao para infinito para phi}
									The length of the orbits of the vector field $X$ forms an unbounded set of the real numbers.
									\end{fact}
									\begin{proof} 
									By Lemma \ref{lemma: X nunca e zero em M}, the vector field $X$ never vanishes and so by compactness of $M$, there is $C>0$ such that $\|X(\xi)\|>C$ for every $\xi\in M$. By Proposition \ref{prop: psi>0 orbitas circulos pra psi}, there is sequence of points $\xi_n$ whose periods $\Per(\xi_n)$ tend to infinity. Hence, the lengths of the orbit of $\xi_n$, which is at least $C\cdot  \Per(\xi_n)$, tend to infinity. 
									\end{proof}

				Hence, we have generalized the counterexample: we can produce a new family of counterexamples to the periodic orbit conjecture for any function $\psi$ satisfying conditions \ref{phi eh c infinito} to \ref{phi eh simetrica}.
			
\section{A new example of flow with emergence of positive order} \label{subsection: An example of foliation with positive order of emergence}

In \cite{bergerbochi}, Theorem B, the following result is proved about flows on the annulus $\mathbb{A}\coloneqq \T\times [0,1]$, where $\T\coloneqq \R/\Z$: 

\begin{thmsn}
	There exists a smooth conservative flow $\left(\Phi^{t}\right)_{t}$ on the annulus $\mathbb{A}$ such that for every $t\neq 0$ the emergence of $f=\Phi^{t}$ has maximal order $d=2$: $$\overline{\mathcal{O}}\mathcal{E}_{f}=d.$$
\end{thmsn}

This subsection aims to show that a slight modification of their example can be viewed as a factor of the generalization of Epstein--Vogt's example we have just constructed. To do so, we need to construct a flow $(\Phi^{t})_{t}$ with maximal order of emergence on the annulus $A$ defined in (\ref{definicao do anel}) p. \pageref{definicao do anel} with $\varphi\colon \R^{2} \to \R$ satisfying conditions \ref{phi eh c infinito} to \ref{phi eh simetrica} in Subsection \ref{section: Epstein-Vogt's counterexample}, where $\Phi^{t}$ is the flow associated with the Hamiltonian $\psi^{2}/2$: $$\begin{cases}
	\dot{x}=\psi\cdot \frac{\partial \psi}{\partial y}\\
	\dot{y}=-\psi\cdot \frac{\partial \psi}{\partial x}
\end{cases},$$ for $\psi=\varphi\cdot\psi_{0}$. This will enable us to carry out the same construction to define, afterward, a flow without singularities on a $4$-manifold. We recall that $\psi_{0}$ is the polynomial function defined on (\ref{psi_0}), p. \pageref{psi_0}, that vanishes on the boundary of the domain $D$.

\begin{thm}\label{teorema jairo pierre eh epstein vogt}
	There exists a smooth Hamiltonian
	flow $\left(\Phi^{t}\right)_{t}$ on the annulus $A$ defined in (\ref{definicao do anel}) such that: \begin{enumerate}[label=\((\alph*)\)]
		\item \label{Phi tem ordem de emergencia maximal} the order of emergence of $\Phi$ is maximal, i.e., $\overline{\mathcal{O}}\mathcal{E}_{\Leb}(\Phi)=2$;
		\item \label{H e hamiltoniano} the Hamiltonian $H$ of $\Phi^{t}$ is of the form $H=\frac{\psi^{2}}{2}$, for some function $\psi$ satisfying conditions \ref{phi eh c infinito} to \ref{phi eh simetrica} on Subsection \ref{section: Epstein-Vogt's counterexample}.
	\end{enumerate}     
\end{thm}

As already stated, the theorem will follow from an adaptation of Theorem B of \cite{bergerbochi}. Given $ \zeta\colon [0,1]\to \R$ a $C^{\infty}$ function, we define the rotation $R^{t}_{\zeta}\colon \mathbb{A}\to \mathbb{A}$ by $R_{ \zeta}^{t}(\theta,\rho)=(\theta+t\cdot  \zeta(\rho),\rho)$. To start the proof, we need a proposition that allows us to create, for each scale $\eps>0$, a new flow $\Psi^{t}_{n,\eps}$ on $\mathbb{A}_{n}\coloneqq \T \times [2^{-n},2^{-n+1}]$, conjugated to the rotation $R^{t}_{ \zeta}$ on $\mathbb{A}$, with emergence bounded from below by an exponential factor. Observe that, since \(R_\zeta^t\) preserves the second coordinate, each horizontal annulus $\mathbb{A}_n$ is invariant by \(R_\zeta^t\). 

\begin{prop}\label{emergencia em cada escala}
	For every natural number $n\geq 1$, there exists a constant $C_{n}>0$ such that for every $\eps>0$, there exists a smooth symplectomorphism $h_{n,\eps}\colon \mathbb{A}_{n}\to \mathbb{A}_{n}$ such that, for every $C^{\infty}$ mapping $ \zeta\colon [0,1]\to \R$ with non-vanishing first and second derivatives, we have that the flow $\Psi^{t}_{n,\eps}\coloneqq h^{-1}_{n,\eps}\circ R^{t}_{ \zeta}\circ h_{n,\eps}$ satisfies: $$\mathcal{E}_{\mathrm{Leb}_{n}}(\Psi^{t}_{n,\eps})(\eps)\geq \exp(C_{n}\cdot \eps^{-2}),$$ where $\mathrm{Leb}_{n}\coloneqq 2^{n}\cdot \restr{\mathrm{Leb}}{\A_{n}}$. Furthermore, $h_{n,\eps}$ is the identity near the boundary of $\mathbb{A}_{n}$, satisfies the symmetry condition $h_{n,\eps}\circ \sigma_{k}=\sigma_{k}\circ h_{n,\eps}$, for $k=1,2$ with $\sigma_{1}(\theta,\rho)\coloneqq \left(-\theta,\rho\right)$ and $\sigma_{2}(\theta,\rho)\coloneqq \left(-\theta+\frac12,\rho\right)$, where $(\theta, \rho)\in \mathbb{A}_{n}$.

	\end{prop}
	\begin{proof}
The proof of this proposition will rely on the following:

\begin{prop}[{\cite[Prop. 4.2]{bergerbochi}}] \label{prop: prop 4.2 pierre-jairo}
	There exists a constant $C>0$ such that for every $\eps>0$, there exists a smooth symplectomorphism $h_{\eps}\colon \mathbb{A}\to \mathbb{A}$ such that, for every $C^{\infty}$ mapping $ \zeta\colon [0,1]\to \R$ with non-vanishing first derivative and for every $t\neq 0$, the flow $\Psi^{t}_{\eps}\coloneqq h^{-1}_{\eps}\circ R^{t}_{ \zeta}\circ h_{\eps}$ satisfies: $$\mathcal{E}_{\mathrm{Leb}}(\Psi^{t}_{\eps})(\eps)\geq \exp(C\cdot  \eps^{-2}).$$ Furthermore, $h_{\eps}$ is the identity near the boundary of $\mathbb{A}$ and on the set $\{(\theta,\rho)\in \A \mid \theta=0\}$.
\end{prop}

\begin{figure}[h]
	\centering
	\IfFileExists{h_n,eps.jpeg}{\includegraphics[width=0.995\textwidth]{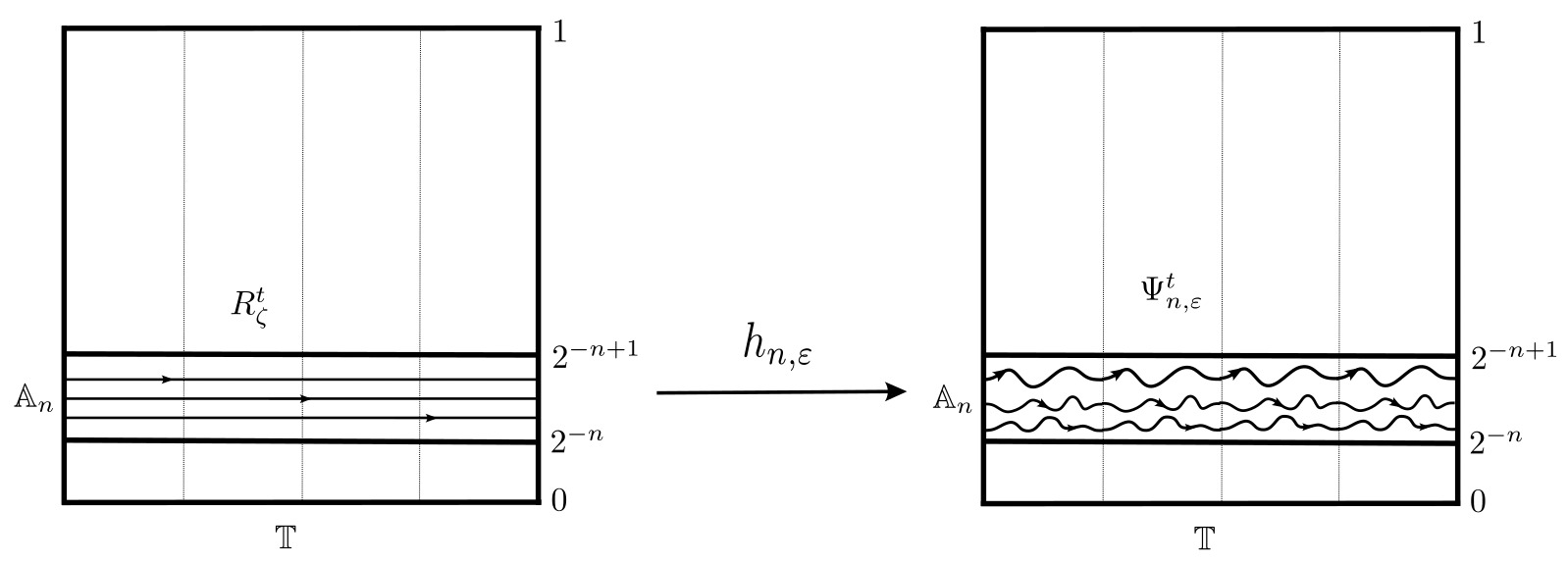}}{\fbox{Missing figure: \texttt{h\_n,eps.jpeg}}}
	\caption{The conjugacy by $h_{n,\eps}$ seen at $\mathbb{A}_{n}$.}
	\label{fig: h_eps}
\end{figure}

Since $h_\eps$ coincides with the identity near $\{\theta=0\}$, by identifying $\T=\R/\Z$ with the fundamental domain $[-1/2,1/2]$ and setting

\[
J\colon (\theta,\rho) \in [0,1/4]\times [0,1]\mapsto (4\theta,\rho)\in \A
\] we can define: 

\[
\tilde{h}_{\eps}\colon (\theta, \rho)\in \A \mapsto \left\{ \begin{array}{cl}
	J^{-1}\circ h_{\eps}\circ J& \ \text{on} \ Q_{1}\coloneqq[0,1/4]\times [0,1]\\
		 \sigma_{1}\circ J^{-1}\circ h_{\eps}\circ J\circ \sigma_{1}& \ \text{on} \ Q_{2}\coloneqq[-1/4,0]\times [0,1]\\
		 \sigma_1 \circ \sigma_2 \circ J^{-1}\circ h_{\eps}\circ J\circ \sigma_1 \circ \sigma_2& \ \text{on} \ Q_{3}\coloneqq[-1/2,-1/4]\times [0,1] \\
		 \sigma_2 \circ J^{-1}\circ h_{\eps}\circ J \circ \sigma_2 & \ \text{on} \ Q_{4}\coloneqq[1/4,1/2]\times [0,1]
	\end{array}
	\right\}
\]

We notice that $\tilde{h}_\eps$ is a symplectomorphism that commutes with the involutions $\sigma_{1}$ and $\sigma_{2}$. 

Now, for any smooth $\eta:[0,1]\to\R$, set
\[
        \tilde\Psi^t_{\eps,\eta}
        :=
        \tilde h_\eps^{-1}\circ R_\eta^t\circ \tilde h_\eps.
\]

\begin{fact}
For every such $\eta$ with non-vanishing first derivative and for every
$t\neq0$, the emergence of $\tilde\Psi^t_{\eps,\eta}$ at scale $\eps$ satisfies
\[
        \mathcal E_{\mathrm{Leb}}(\tilde\Psi^t_{\eps,\eta})(\eps)
        \geq
        \exp\left(\frac{C}{16}\eps^{-2}\right).
\]
\end{fact}

\begin{proof}
	This follows from the symmetry of $\tilde{h}_{\eps}$ in Proposition \ref{prop: prop 4.2 pierre-jairo}.
\end{proof}

Now, given $\eps>0$, we conjugate $\tilde h_{\eps}$ by the affine rescaling
$g_{n}\colon \mathbb A_{n}\to \mathbb A$ given by $g_{n}(\theta,\rho)=(\theta,2^{n}\rho-1)$, which maps $[2^{-n},2^{-n+1}]$ to $[0,1]$. Define $h_{n,\eps}\colon \A_{n}\to \A_{n}$ as $h_{n,\eps}:=g_{n}^{-1}\circ \tilde h_{\eps}\circ g_{n}$.

Although $R_{\zeta}^{t}$ preserves $\mathbb A_{n}$, its expression in the
rescaled coordinate is a rotation with a rescaled speed. More precisely,
define
\begin{equation*}
\zeta_{n}(s):=\zeta\bigl(2^{-n}(s+1)\bigr),
\qquad s\in[0,1].
\end{equation*}
Then
\begin{equation*}
g_{n}\circ R_{\zeta}^{t}\circ g_{n}^{-1}
=
R_{\zeta_{n}}^{t}.
\end{equation*}
Therefore, if
\begin{equation*}
\Psi_{n,\eps}^{t}:=h_{n,\eps}^{-1}\circ R_{\zeta}^{t}\circ h_{n,\eps},
\end{equation*}
then
\begin{equation*}
\Psi_{n,\eps}^{t}
=
g_{n}^{-1}\circ
\bigl(\tilde h_{\eps}^{-1}\circ R_{\zeta_{n}}^{t}\circ \tilde h_{\eps}\bigr)
\circ g_{n}.
\end{equation*}
Thus, in the rescaled annulus $\mathbb A$, the correct speed is $\zeta_n$.
Since $\zeta_{n}'$ and $\zeta_{n}''$ do not vanish whenever $\zeta'$ and
$\zeta''$ do not vanish, the previous construction applies to $\zeta_n$.

We claim that 
$$\mathcal{E}_{\mathrm{Leb}_{n}}(\Psi^{t}_{n,\eps})(\eps)\geq \mathcal{E}_{\mathrm{Leb}}(\tilde{\Psi}^{t}_{\eps,\zeta_n})(2^{n}\eps)\geq \exp\left(\frac{C}{16}\cdot  2^{-2n}\eps^{-2}\right),$$ that is 
$$\mathcal{E}_{\mathrm{Leb}_{n}}(\Psi^{t}_{n,\eps})(\eps)\geq \exp(C_{n}\cdot  \eps^{-2}),$$ for $C_{n}=2^{-2n}\cdot \frac{C}{16}$. To give a complete proof of this claim, we need a general proposition that explains how the ergodic decomposition of a measure behaves through conjugacy:

\begin{propsn}[{\cite[Lemma 3.11]{bergerbochi}}]
	Let $X$ and $Y$ be compact metric spaces, $f\colon X\to X$ and $g\colon Y\to Y$ continuous, and $h\colon X\to Y$ a continuous semi-conjugacy between $f$ and $g$, i.e., $h\circ f=g\circ h$. 
	\[
	\begin{array} {rcccl}
		& &f& &\\
		& X  &\to & X &\\
		h&\downarrow &  &\downarrow &h\\
		& Y  &\to & Y &\\
		& &g& &
	\end{array} 
	\]
	Then, $h_{\ast}(\mathcal{M}_{f}(X))\subset \mathcal{M}_{g}(Y)$, $h_{\ast}(\mathcal{M}_{f}^{\mathrm{erg}}(X))\subset \mathcal{M}_{g}^{\mathrm{erg}}(Y)$, and for all $\mu \in \mathcal{M}_{f}(X)$ we have: $$\mathrm{e}_{\ast}^{g}(h_{\ast}\mu)=(h_{\ast})_{\ast}(\mathrm{e}_{\ast}^{f}(\mu)).$$ If there is no confusion, we write $h_{\ast\ast}$ instead of $(h_{\ast})_{\ast}$. 
\end{propsn}

Note that $\mathrm{Leb}=(g_{n})_{\ast}\mathrm{Leb}_{n}$. Let $\widehat{\mu}_{n}\coloneqq\mathrm{e}^{\Psi^{t}_{n,\eps}}_{\ast}(\mathrm{Leb}_{n})$ and $\widehat{\mu}\coloneqq \mathrm{e}^{\tilde \Psi^{t}_{\eps,\zeta_n}}_{\ast}(\mathrm{Leb})$. By the above proposition we have: $(g_{n})_{\ast\ast}\widehat{\mu}_{n}=\widehat{\mu}$. Since $g_{n}$ is $2^{n}$-Lipschitz, using Proposition \ref{quantizacao lipschitz}, the fact that $\Leb$ is invariant by $\tilde \Psi_{\eps,\zeta_n}^{t}$, and Proposition \ref{prop: quantizacao vs emeergencia}, we get: $$
\mathcal{E}_{\mathrm{Leb}_{n}}(\Psi^{t}_{n,\eps})(\eps)=\mathcal{Q}_{\widehat{\mu}_{n}}(\eps) \geq \mathcal{Q}_{\widehat{\mu}}(2^{n}\eps)=\mathcal{E}_{\mathrm{Leb}}(\tilde \Psi^{t}_{\eps,\zeta_n})(2^{n}\eps)\geq \exp\left(\frac{C}{16}\cdot 2^{-2n}\eps^{-2}\right),$$ 
i.e., $\mathcal{E}_{\mathrm{Leb}_{n}}(\Psi^{t}_{n,\eps})(\eps)\geq \exp\left(C_{n}\cdot \eps^{-2}\right)$, where $C_{n}=\frac{C}{16}\cdot 2^{-2n}$.
\end{proof}

Let $\check{A}$ denote the annulus obtained from $A$ by removing its outer boundary component (the octagon). The following lemma was written in collaboration with Pierre Berger.

\begin{lemma} \label{lemma: construcao da F}
	There exists a symplectomorphism $F: \T\times (0,1]\to \check{A}$ satisfying 
	\[
	F\circ \sigma_1 =S_1 \circ F \ \text{and} \ F\circ \sigma_2=S_2 \circ F,
	\]
	where $S_{1}(x,y)=(-x,y)$ and $S_{2}(x,y)=(x,-y)$.
\end{lemma}	
\begin{proof} 
    Recall that $A$ is the union of the energy levels of $\psi$ between $0$ and $1$. 
	Let $G: (0,1]\to (0,1]$ be the diffeomorphism whose derivative has its inverse equal to the density of $\psi_\ast \Leb_{(0,1]}$. Note the $\tilde \psi\coloneqq G\circ \psi$ satisfies that $\tilde \psi_{\ast}\Leb_{(0,1]}=\Leb_{(0,1]}$.
	  
	Hence with $\Per(x,y)$ the period of $(x,y)$ for the flow $(\phi^{t}_{\tilde \psi})_{t}$, the following map is a symplectomorphism between $\T\times (0,1]$ and $\check{A}$: 
	
	\[
     F: (\theta,\rho)\mapsto \phi_{\tilde \psi}^{\theta\cdot \Per(m(\rho))}(m(\rho)) 
	\]
	with $m(\rho)$ the point of $\check A$ with $\tilde \psi$-energy $\rho$ and $x$-coordinate $0$.
	
	Moreover, since $\tilde \psi$ is equivariant under $(x,y)\mapsto (-x,y)$ and $(x,y)\mapsto (x,-y)$, the map $F$ satisfies the claimed symmetry conditions. 
\end{proof}

Set $A_{n}\coloneqq F(\mathbb{A}_{n})$ and set \( \Leb_{A_n}:=\frac{\Leb|_{A_n}}{\Leb(A_n)}\). The family $\{A_{i}\}_{i\geq 1}$ forms a partition of $\check{A}$. A direct corollary of the previous Proposition \ref{emergencia em cada escala} is the following version for $A_n$ in place of $\mathbb{A}_n$:

\begin{coro}\label{corolario da modificacao da prop 7}
For every natural number $n\geq 1$, there exists a constant $\widehat{C}_{n}>0$ such that for every $\eps>0$, there exists a smooth symplectomorphism $\widehat{h}_{n,\eps}\colon A_{n}\to \A_{n}$ such that, for every $C^{\infty}$ mapping $ \zeta\colon [0,1]\to \R$ with non-vanishing first and second derivatives, we have that the flow $\widehat{\Psi}^{t}_{n,\eps}\coloneqq \widehat{h}^{-1}_{n,\eps}\circ 
R^{t}_{ \zeta}\circ  \widehat{h}_{n,\eps}$ satisfies: 
\[
\mathcal{E}_{\mathrm{Leb}_{A_n}}\big(\widehat\Psi^t_{n,\eps}\big)(\eps)
\;\ge\; \exp\!\left(\widehat C_n\cdot\eps^{-2}\right).
\]
Furthermore, $\widehat{h}_{n,\eps}$ coincides with $F^{-1}$ near the boundary of $A_{n}$, and satisfies the symmetry conditions 
\[
\widehat{h}_{n,\eps}\circ S_{k}=\sigma_{k}\circ \widehat{h}_{n,\eps}\qquad (k=1,2).
\]
\end{coro}

\begin{proof}

By regularity of $F$, the restriction
$F|_{\mathbb A_n}\colon \mathbb A_n\to A_n$
is bi-Lipschitz; denote a bi-Lipschitz constant by $\kappa_n$. Define
\[
\widehat h_{n,\eps}\coloneqq h_{n,\eps}\circ \big(\,F^{-1}\big|_{A_n}\big).
\]
Since $h_{n,\eps}$ and $F^{-1}$ are symplectomorphisms, so is $\widehat h_{n,\eps}$. Moreover $\widehat h_{n,\eps}$ coincides with $F^{-1}$ near the boundary of $A_n$. Using the bi-Lipschitz bound and Proposition \ref{emergencia em cada escala}, there exists a constant $\widehat C_n \coloneqq \frac{C}{16}\cdot2^{-2n}\cdot  \kappa_n^{-2}$ such that 
\[
\mathcal{E}_{\mathrm{Leb}_{A_n}}\big(\widehat\Psi^t_{n,\eps}\big)(\eps)
\;\ge\; \exp\!\left(\widehat C_n\cdot\eps^{-2}\right).
\]
Finally, the symmetry and boundary properties follow from those of $h_{n,\eps}$ and $F$ (Lemma \ref{lemma: construcao da F}).
\end{proof}

	We now use Proposition \ref{emergencia em cada escala} to prove Theorem \ref{teorema jairo pierre eh epstein vogt}:

	\begin{proof}[Proof of Theorem \ref{teorema jairo pierre eh epstein vogt}]
		Let $(\alpha_{i})_{i}$ be a sequence of positive real numbers converging monotonically to 2. We claim that we can choose $(\eps_{i})_{i}$ and define a Hamiltonian flow $\Phi^{t}=\Phi^{t}_{\Omega,\eps}\colon A\to A$ such that for each $i\geq 1$,  $$\mathcal{E}_{\mathrm{Leb}}(\Phi^{t}_{\Omega,\eps})(\eps_{i})\geq \exp(\eps_{i}^{-\alpha_{i}}).$$ More precisely, we claim that, by choosing the sequence $(\eps_{i})_{i}$ converging to zero and such that $\widehat{C}_{i}\eps_{i}^{-2}\geq \eps_{i}^{-\alpha_{i}}$, and by picking $\Phi^{t}_{\Omega,\eps}$ as the flow associated to the Hamiltonian $H_{\Omega, \eps}\colon A \to \R$ given by 
		\[
		H_{\Omega,\varepsilon}(x,y)=
		\begin{cases}
			\Omega\circ h \circ F^{-1}(x,y),& (x,y)\in\check{A},\\[4pt]
			0,&\text{otherwise},
		\end{cases}
		\] we are done. The map $h\colon \A\to \A$ above is defined by $\restr{h}{\A_{i}}=h_{i,\eps_{i}}$, for all $i\geq 1$
		and $\Omega\colon \A\to \R$ is the Hamiltonian of $R_{\zeta}^{t}$, which is a smooth function such that $\Omega(\theta,\rho)=\Omega(\rho)$, i.e., $\Omega$ depends only on the second coordinate, and such that $\Omega^{\prime}$ and $\Omega^{\prime\prime}$ do not vanish on $(0,1]$ (and therefore $\nabla H_{\Omega,\eps}$ also does not vanish). Since $F$ satisfies the equivariance conditions demanded in Lemma \ref{lemma: construcao da F}, $h$ commutes with $\sigma_{k}$, and $\Omega$ does not depend on $\theta$, we obtain $$H_{\Omega,\eps}(x,y)=H_{\Omega,\eps}(-x,y)=H_{\Omega,\eps}(x,-y),$$ for every $(x,y)\in A$. Moreover, we claim that by choosing well $\Omega$, $H_{\Omega,\eps}$ is smooth. To see that, we first observe that when $i\to \infty$, the derivative of $F_{i}=\restr{F}{\A_{i}}$ explodes: otherwise we would get a Lipschitz map between $A$ and $\A$. This will also happen with $h$. So we need to ask $\Omega$ to decay fast enough near the boundary of $A$ to compensate for this behavior. This can be achieved by similar choices as the ones done in \cite{bergerbochi}: if we demand that $\Omega(0)=\Omega^{\prime}(0)=\Omega^{\prime\prime}(0)=0$ and $\Omega^{\prime\prime\prime}=j$, where $j\colon [0,1]\to \R$ is a function of the form $$j(\rho)\coloneqq \begin{cases}
			\eta_{i}\cdot b(2^{i}\cdot \rho -1) , & \text{if } \rho\in [2^{-i},2^{-i+1}] \\
			0, & \text{if }\rho=0
		\end{cases}, $$
		 where $b\colon \R\to \R$ is a smooth bump function that vanishes on $(-\infty,0]$ and $[1,+\infty)$ and it is positive on $(0,1)$, and $(\eta_{i})_{i}$ is a sequence of positive real numbers converging very rapidly to 0. By choosing the sequence well, we can have that the $C^{i}$-norm of the restriction of $\Omega$ to $[2^{-i},2^{-i+1}]$ is small. This shows that we can choose inductively the sequence $(\eta_{i})_{i}$ so that the $C^{i}$-norm of $H_{\Omega, \eps}$ goes to zero when $(x,y)$ approaches the outer boundary of $A$ for every $i$, proving that $H_{\Omega,\eps}$ is smooth.

		Now, observe that if we take $\Phi^{t}_{\Omega,\eps}$ as the Hamiltonian flow of $H_{\Omega,\eps}$ as above, we get for every $i\geq 1$: \begin{align*}
			\mathcal{E}_{\mathrm{Leb}}(\Phi^{t}_{\Omega,\eps})(\eps_{i})&\geq \mathcal{E}_{\mathrm{Leb}}(\Psi^{t}_{i,\eps_{i}})(\eps_{i})
			\\ &\geq \exp(\widehat{C}_{i}\cdot \eps_{i}^{-2})
			\\&\geq \exp (\eps_{i}^{-\alpha_{i}}),
		\end{align*} where the first inequality comes by definition of $h$ and the second by Corollary \ref{corolario da modificacao da prop 7}.
		
		Thus, for each $i\geq 1$, \begin{align*}
			\frac{\log{\log{\mathcal{E}_{\mathrm{Leb}}(\Phi^{t})(\eps_{i})}}}{-\log{\eps_{i}}}
			\geq \frac{\log{\log{\exp(\eps_{i}^{-\alpha_{i}})}}}{-\log{\eps_{i}}}=\alpha_{i}.
		\end{align*} Since $\alpha_{i}$ converges to $2$ as $i\to +\infty$, this proves that $\overline{\mathcal{O}}\mathcal{E}_{f}=2$, for every $f=\Phi^{t}$ with $t\neq 0$.

		This proves that $\Phi^{t}=\Phi^{t}_{\Omega,\eps}$ is a smooth Hamiltonian flow such that the function $\psi\coloneqq H_{\Omega,\eps}$ is divisible by $\psi_{0}$, satisfies the conditions \ref{phi eh c infinito} to \ref{phi eh simetrica}, and for every $t\neq 0$ the order of emergence of $f=\Phi^{t}$ is maximal, i.e., $\overline{\mathcal{O}}\mathcal{E}_{f}=2$.
	\end{proof}

Now we apply the construction of Subsection \ref{section: Epstein-Vogt's counterexample} to obtain a regular flow with positive order of emergence and prove Theorem~\ref{thm:main-lower-bound}:
					
					\begin{thm} \label{epstein vogt com emergencia}
						There exists a compact $4$-manifold $M$ and a $C^{\infty}$ regular flow $(\Theta^{t})_{t}$ on $M$ such that for every $t\neq 0$, the order of emergence of $f=\Theta^{t}$ is at least 2: $$\overline{\mathcal{O}}\mathcal{E}_{f}\geq 2.$$
					\end{thm}

					\begin{proof}[Proof of Theorem \ref{epstein vogt com emergencia}]
						The theorem follows from the fact that, if $M$ and $\Theta^{t}$ are the manifold and flow obtained from the construction made in Subsection \ref{section: Epstein-Vogt's counterexample} with the function $\psi$ furnished by item \ref{H e hamiltoniano} of Theorem \ref{teorema jairo pierre eh epstein vogt}, then the flow $\Phi^{t}$ is the projection onto the first two coordinates of the flow $\Theta^{t}$. 
						
						Since this projection is $1$-Lipschitz, Lemma \ref{lemma lipschitz preserva emergencia} implies that $$\overline{\mathcal{O}}\mathcal{E}_{f}\geq \overline{\mathcal{O}}\mathcal{E}_{g}=2,$$ where $f=\Theta^{t}$ and $g=\Phi^{t}$, for $t\neq 0$.
					\end{proof}
					
					A final remark: regular flows with positive order of emergence were already achievable from \cite{bergerbochi} and \cite{berger22}. There, the authors provide examples of maps on compact manifolds with positive order of emergence. By taking the suspension flow of those examples, we would get the desired examples. 
					
					Here, however, we give a new example of regular flow with positive order of emergence that \textbf{does not} come from a suspension process. Indeed, if $(\Theta^t)_{t}$ were a suspension flow, it would admit a global section and a Poincaré return map whose every orbit is periodic while the periods are unbounded. This contradicts Montgomery's theorem \cite{montgomery}, which states that if a homeomorphism $f$ on a connected manifold has all orbits periodic, then $f$ itself is periodic: there exists $N\in\N$ with $f^N=\mathrm{Id}$.


\bibliographystyle{amsalpha}
\bibliography{high_emergence_for_periodic_flows}

\end{document}